\documentclass[a4paper]{amsart}

\usepackage{amsmath,amsfonts,amssymb,amsthm} 
\usepackage{mathrsfs} 
\usepackage[shortlabels]{enumitem} 
\usepackage{caption} 
\usepackage[T1]{fontenc} 

\usepackage{mathtools} 
\mathtoolsset{centercolon} 

\usepackage{tikz} 
\usetikzlibrary{cd} 

\usepackage{todonotes}

\usepackage{hyperref} 

\theoremstyle{definition}
\newtheorem{definition}{Definition}[section]

\newtheorem{notation}[definition]{Notation}

\theoremstyle{plain}
\newtheorem{lemma}[definition]{Lemma}
\newtheorem{theorem}[definition]{Theorem}

\theoremstyle{remark}
\newtheorem{remark}[definition]{Remark}

\newcommand{\bA}{\mathbb{A}}

\newcommand{\bC}{\mathbb{C}}

\newcommand{\bF}{\mathbb{F}}

\newcommand{\bP}{\mathbb{P}}
\newcommand{\bQ}{\mathbb{Q}}
\newcommand{\bR}{\mathbb{R}}

\newcommand{\bW}{\mathbb{W}}

\newcommand{\bZ}{\mathbb{Z}}

\newcommand{\calE}{\mathcal{E}}

\newcommand{\calL}{\mathcal{L}}
\newcommand{\calM}{\mathcal{M}}

\newcommand{\calO}{\mathcal{O}}

\newcommand{\calX}{\mathcal{X}}

\DeclareMathOperator{\Cl}{Cl}

\DeclareMathOperator{\Hom}{Hom}

\DeclareMathOperator{\NS}{NS}
\DeclareMathOperator{\Pic}{Pic}
\DeclareMathOperator{\Proj}{Proj}
\DeclareMathOperator{\rank}{rank}

\numberwithin{figure}{section}
\numberwithin{table}{section}

\begin{document}

\title[Normal Forms and Tyurin Degenerations of K3 Surfaces]{Normal Forms and Tyurin Degenerations of K3 Surfaces Polarised by a Rank 18 Lattice}

\author[C. F. Doran]{Charles F. Doran}
\address{Department of Mathematical and Statistical Sciences, 632 CAB, University of Alberta,  Edmonton, AB, T6G 2G1, Canada. \newline \indent Bard College, Annandale-on-Hudson, NY 12571, USA. \newline \indent Center of Mathematical Sciences and Applications, Harvard University, 20 Garden Street, Cambridge, MA 02138, USA.}
\email{charles.doran@ualberta.ca}
\thanks{Charles F. Doran (Alberta/Bard College/CMSA Harvard) was supported by the Natural Sciences and Engineering Research Council of Canada (NSERC)}

\author[J. Prebble]{Joseph Prebble}
\address{Department of Mathematical Sciences, Loughborough University, Loughborough, Leicestershire, LE11 3TU, United Kingdom.}
\email{J.Prebble@lboro.ac.uk}
\thanks{Joseph Prebble (Loughborough) was partially supported by a PhD studentship funded by the Engineering and Physical Sciences Research Council (EPSRC)}

\author[A. Thompson]{Alan Thompson}
\address{Department of Mathematical Sciences, Loughborough University, Loughborough, Leicestershire, LE11 3TU, United Kingdom.}
\email{A.M.Thompson@lboro.ac.uk}
\thanks{Alan Thompson (Loughborough) was supported by Engineering and Physical Sciences Research Council (EPSRC) New Investigator Award EP/V005545/1.}

\begin{abstract}We study projective Type II degenerations of K3 surfaces polarised by a certain rank 18 lattice, where the central fibre consists of a pair of rational surfaces glued along a smooth elliptic curve. Given such a degeneration, one may construct other degenerations of the same kind by flopping curves on the central fibre, but the degenerations obtained from this process are not usually projective. We construct a series of examples which are all projective and which are all related by flopping single curves from one component of the central fibre to the other. Moreover, we show that this list is complete, in the sense that no other flops are possible. The components of the central fibres obtained include weak del Pezzo surfaces of all possible degrees. This shows that projectivity need not impose any meaningful constraints on the surfaces that can arise as components in Type II degenerations.
\end{abstract}


\date{}
\maketitle

\section{Introduction}

A Tyurin degeneration is a semistable degeneration of K3 surfaces $\calX \to \Delta$, where $\Delta$ denotes the complex unit disc, with central fibre $X_0 = Y_1 \cup_D Y_2$ consisting of a pair of rational surfaces $Y_1$, $Y_2$ glued along a smooth elliptic curve $D$ which is anticanonical in each. They are the simplest examples of Type II degenerations of K3 surfaces.

In this paper we will be interested in studying projective Tyurin degenerations, i.e. ones where the morphism $\calX \to \Delta$ is projective. If one assumes that the smooth fibres of $\calX$ are polarised by an appropriate ample divisor, then results of Shepherd-Barron \cite{bgd4} show that projectivity can always be arranged by flopping some curves between the components of the central fibre $X_0$. 

If one further assumes that the smooth K3 fibres of $\calX$ are lattice polarised by a lattice $L$, Type II degenerations occur over $1$-cusps in the Baily-Borel compactification of the moduli space of $L$-polarised K3 surfaces. These $1$-cusps have a lattice theoretic description: they correspond to isotropic sublattices of rank $2$ in the orthogonal complement $L^{\perp}$ of $L$ inside the K3 lattice $H \oplus H \oplus H \oplus E_8 \oplus E_8$. Thus, given a Tyurin degeneration of $L$-polarised K3 surfaces, one may associate to it the lattice theoretic data of a rank $2$ isotropic sublattice of $L^{\perp}$.

The question that this paper seeks to answer is as follows. To what extent is the geometry of the central fibre $X_0 = Y_1 \cup_D Y_2$ of a projective Tyurin degeneration determined by this lattice theoretic data? In particular, can one deduce anything about the geometry of the rational surfaces $Y_i$, such as their del Pezzo degrees $K_{Y_i}^2$?
\smallskip

There is some reason to expect that something like this might hold. Indeed, given two Tyurin degenerations that are isomorphic over the punctured disc, it is known that they must differ by a sequence of flops. However, the process of flopping a curve is an analytic operation, and doing so will often destroy the projectivity of the Tyurin degeneration. Indeed, in the case of $\langle 2 \rangle$-polarised K3 surfaces, Friedman \cite[Section 5]{npgttk3s} has shown that projectivity imposes a very strong restriction, and that projective Tyurin degenerations of $\langle 2 \rangle$-polarised K3 surfaces are essentially unique up to labelling of the components. In other words, in the $\langle 2 \rangle$-polarised case, the birational geometry of the central fibre of a projective Tyurin degeneration is completely determined by the corresponding $1$-cusp in the Baily-Borel compactification, which is described by purely lattice theoretic data.

Interestingly, as we will show in this paper, such a result does not hold for general lattice polarisations. We study Tyurin degenerations of K3 surfaces polarised by the rank $18$ lattice $M = H \oplus E_8 \oplus E_8$. The moduli space of $M$-polarised K3 surfaces is well-understood (see, for example, \cite{milpk3s,nfk3smmp}). Its Baily-Borel compactification contains a single $1$-cusp, corresponding to the unique (up to isometry) rank $2$ isotropic sublattice of $M^{\perp} \cong H \oplus H$. We show that, given a projective Tyurin degeneration of $M$-polarised K3 surfaces $\calX \to \Delta$, there are several other projective Tyurin degenerations that differ from $\calX$ by flops. In fact we prove more than this: \emph{all} possible flops of $(-1)$-curves in the central fibre $X_0$ give rise to new projective Tyurin degenerations. This answers the question above in a maximally negative way: the lattice theory associated to such a projective Tyurin degeneration provides essentially no constraints upon which surfaces $Y_1$ and $Y_2$ can arise. In particular, the del Pezzo degrees  $K_{Y_i}^2$ of the components $Y_i$ in our examples take all possible values $-9 \leq K_{Y_i}^2 \leq 9$.

The approach to proving this result is explicit: we construct projective models for all of the Tyurin degenerations. As a by-product, we obtain several simple projective normal forms for $M$-polarised K3 surfaces which have not previously appeared in the literature.
\smallskip

The motivation for this paper comes from mirror symmetry. It was essentially known to Dolgachev \cite{mslpk3s} (see also \cite[Section 4]{mstdfcym}) that there is a mirror symmetric correspondence between Type II degenerations $X \rightsquigarrow X_0$ of a K3 surface $X$ and genus $1$ fibrations $\check{X} \to \bP^1$ on its mirror K3 surface $\check{X}$. Explicitly this works as follows. Given an $L$-polarised K3 surface $X$ and an isotropic vector $e \in L^{\perp}$, the mirror $\check{X}$ is defined to be a K3 surface polarised by the mirror lattice $\check{L} = e^{\perp}/\bZ e$, where the orthogonal complement of $e$ is taken inside $L^{\perp}$. Such vectors $e$ correspond to $0$-cusps in the Baily-Borel compactification of the moduli space of $L$-polarised K3 surfaces. If the $1$-cusp corresponding to a Type II degeneration is incident to such a $0$-cusp, this gives an embedding of $e$ into the rank $2$ isotropic sublattice $I$ corresponding to the $1$-cusp. This in turn gives rise to a class of self-intersection $0$ in the mirror lattice $\check{L}$, which induces a genus $1$-fibration on the mirror K3.

The DHT conjecture \cite{mstdfcym} extends this idea by suggesting that, if $X_0 = Y_1 \cup_D Y_2$ is the central fibre in a projective Tyurin degeneration, then there exists a splitting $\bP^1 = \Delta_1 \cup_{\gamma} \Delta_2$ of the base of the mirror fibration along a curve $\gamma$, such that the restriction of the fibration to $\Delta_i$ is the Landau-Ginzburg model of $Y_i$.

The results of this paper show that information about the geometry of a projective Tyurin degeneration cannot necessarily be inferred from the lattice theoretic data of the polarising lattice $L$ and an isotropic flag $e \in I \subset L^{\perp}$. Consequently, we expect that this data alone is also probably insufficient to determine the splitting of the base $\bP^1$ on the mirror. We therefore expect that a proof of the DHT conjecture for K3 surfaces will probably require the introduction of some additional data, in order to uniquely specify the projective Tyurin degeneration and the mirror splitting of $\bP^1$.
\smallskip

Focussing on the Tyurin degenerations of $M$-polarised K3 surfaces studied in this paper, the picture above can be described explicitly as follows. The Baily-Borel compactification of the moduli space of $M$-polarised K3 surfaces contains a unique $0$-cusp, which is incident to the single $1$-cusp. The corresponding mirror K3 surfaces are polarised by the lattice $H$, and a generic member of the mirror family admits a unique elliptic fibration $\check{X} \to \bP^1$ with section and $24$ singular fibres of Kodaira type $\mathrm{I}_1$.

Under the DHT conjecture, our result that an $M$-polarised K3 surface admits many different projective Tyurin degenerations should correspond to the existence of many different splittings of the base $\bP^1$ of this fibration, all of which should satisfy the requirement that the two halves have the properties required of Landau-Ginzburg models. 

There are strong hints that something like this should be true. We suspect that it should be possible to endow the base $\bP^1$ with the structure of an integral affine sphere, with singularities corresponding to the locations of singular fibres, and that this should correspond under mirror symmetry to the integral affine sphere described in \cite[Section 7]{cmek3sspt}. In this description, the possible splitting curves $\gamma \subset \bP^1$ should correspond to vertical lines passing between the singularities in \cite[Figure 5]{cmek3sspt}, and performing a single flop on the Tyurin degeneration should correspond to moving the splitting curve one space to the left or right, corresponding to moving a single $\mathrm{I}_1$ fibre from one side of $\gamma$ to the other.

Importantly, as all of the projective Tyurin degenerations constructed in this paper are polarised by the same lattice $M$ and correspond to the  the same $1$-cusp in the Baily-Borel compactification, we anticipate that the correct splitting curve $\gamma$ in the mirror cannot be determined by this data alone.
\smallskip

The structure of this paper is as follows. In Section \ref{sec:background} we discuss some necessary background from lattice theory, along with results about the geometry and moduli of $M$-polarised K3 surfaces. 

In Section \ref{sec:Tyurin} we begin by introducing some background on Tyurin degenerations. We then prove our first main result (Theorem \ref{thm:analyticclassification}) which shows that, up to labelling, the central fibre of any (analytic) Tyurin degeneration of $M$-polarised K3 surfaces must be one of $11$ distinct birational types. Note, in particular, that the set of possible components $Y_i$ occurring in such degenerations includes rational surfaces of all del Pezzo degrees $-9 \leq K_{Y_i}^2 \leq 9$, along with the Hirzebruch surfaces $\bF_1$ and $\bF_2$ (which is a deformation of $\bP^1 \times \bP^1$) which both have degree $8$. 

Next we turn our attention to the projective setting, where we give two different explicit approaches to the construction of projective normal forms for $M$-polarised K3 surfaces; we call these the \emph{linear system approach} and the \emph{toric approach}. We conclude the section by showing how these two approaches can be used to construct projective Tyurin degenerations of $M$-polarised K3 surfaces.

Section \ref{sec:constructions} constitutes the main body of the paper. In this section we prove our second main result (Theorem \ref{thm:main}), which states that all of the $11$ possible central fibres described in Theorem \ref{thm:analyticclassification} can be realised in projective degenerations. We prove this result by explicitly constructing several projective normal forms for $M$-polarised K3 surfaces, then degenerating them to obtain each of the $11$ possibilities; the constructions are summarised in Table \ref{tab:gammadegenerations}.
\smallskip

\noindent\textbf{Acknowledgments:} The authors would like to thank the anonymous referee for their helpful comments on an earlier version, which outlined an approach to show that our list of possible degenerations is complete. This has been incorporated into the paper and may be found in Section \ref{sec:analytictyurin}.
\smallskip

\noindent\textbf{Data access statement:} No datasets were generated or analysed during this study.

\section{Background on $M$-polarised K3 surfaces} \label{sec:background}

\subsection{Lattice theory}

We begin with some generalities on lattice polarisations and the lattice $M$. Here we follow Alexeev's and Engel's modification of Dolgachev's original formulation; for details we refer the reader to \cite[Section 2]{cmk3s} and \cite[Section 1]{mslpk3s}. Throughout the paper we use $M$ to denote the rank $18$ lattice
\[M := H \oplus E_8 \oplus E_8,\]
where $H$ denotes the usual hyperbolic plane lattice, the even unimodular indefinite lattice of rank $2$, and $E_8$ denotes the negative definite $E_8$ root lattice. 

Fix a vector $h \in M \otimes \bR$ such that $\langle h,h \rangle > 0$ and $h \notin L \otimes \bR$ for any sublattice $L \subset M$ with $\rank(L) < 18 =  \rank(M)$; Alexeev and Engel call such vectors \emph{very irrational}.

\begin{definition}\label{def:Mpolarised} \cite[Definition 2.10]{cmk3s} An \emph{$M$-polarised K3 surface} is a surface $X$ with at worst rational double point (ADE) singularities whose minimal resolution is a smooth K3 surface, along with a lattice embedding $\iota\colon M \hookrightarrow \NS(X)$ such that $\iota(h)$ is ample. \end{definition}

A class $\delta \in M$ with $\langle \delta,\delta\rangle = -2$ is called a \emph{root}. The lattice $M$ contains a set of roots 
\[\Delta(M)^+ := \{ \delta \in M : \langle \delta,\delta \rangle = -2, \ \langle \delta, h \rangle > 0 \}.\]
The fact that $h$ is very irrational implies that $\Delta(M)^+$ is a set of \emph{positive roots} for $M$, in other words:
\begin{itemize}
\item for every root $\delta \in M$, we either have $\delta \in \Delta(M)^+$ or $-\delta \in \Delta(M)^+$, but not both;
\item if $\delta$ is a root that can be written as a non-negative integer combination of classes from $\Delta(M)^+$, then $\delta \in \Delta(M)^+$.
\end{itemize}

From now on $X$ will denote a very general $M$-polarised K3 surface with a fixed choice of polarisation $\iota$. As $X$ is very general, $X$ is a smooth K3 surface and the polarisation $\iota$ induces an isomorphism $\NS(X) \cong M$, under which $\Delta(M)^+$ is identified with the set of classes of effective divisors with self-intersection $(-2)$. As a consequence of this identification, we will often use the same notation to refer to elements of $M$ and classes in $\NS(X)$.

A \emph{simple root} in $\Delta(M)^+$ is one that cannot be written as a non-negative integer combination of other roots from $\Delta(M)^+$; they correspond to classes of irreducible $(-2)$-curves in $\NS(X)$. There are $19$ such roots, which we label $E_0,\ldots,E_{18}$, with Dynkin diagram as in Figure \ref{fig:dualgraph}; in this diagram nodes correspond to roots and two roots $E_i$, $E_j$ are joined if and only if $\langle E_i,E_j \rangle = 1$. 

\begin{figure}
\begin{tikzpicture}[scale=0.7,thick]
\draw (2,0)--(2,1);
\draw (0,1)--(16,1);
\draw(14,0)--(14,1);

\node[above] at (0,1) {$E_1$};
\node[above] at (1,1) {$E_2$};
\node[above] at (2,1) {$E_3$};
\node[above] at (3,1) {$E_4$};
\node[above] at (4,1) {$E_5$};
\node[above] at (5,1) {$E_6$};
\node[above] at (6,1) {$E_7$};
\node[above] at (7,1) {$E_8$};
\node[above] at (8,1) {$E_9$};
\node[above] at (9,1) {$E_{10}$};
\node[above] at (10,1) {$E_{11}$};
\node[above] at (11,1) {$E_{12}$};
\node[above] at (12,1) {$E_{13}$};
\node[above] at (13,1) {$E_{14}$};
\node[above] at (14,1) {$E_{15}$};
\node[above] at (15,1) {$E_{16}$};
\node[above] at (16,1) {$E_{17}$};

\node[left] at (2,0) {$E_0$};
\node[right] at (14,0) {$E_{18}$};

\filldraw[black]
(2,0) circle (3pt)
(14,0) circle (3pt)
(1,1) circle (3pt)
(0,1) circle (3pt)
(2,1) circle (3pt)
(3,1) circle (3pt)
(4,1) circle (3pt)
(5,1) circle (3pt)
(6,1) circle (3pt)
(7,1) circle (3pt)
(8,1) circle (3pt)
(9,1) circle (3pt)
(10,1) circle (3pt)
(11,1) circle (3pt)
(12,1) circle (3pt)
(13,1) circle (3pt)
(14,1) circle (3pt)
(15,1) circle (3pt)
(16,1) circle (3pt)
;

\end{tikzpicture}
\caption{Dual graph of the configuration of simple roots in $M$.}
\label{fig:dualgraph}
\end{figure}

It is easy to show that the $19$ roots $E_i$ span $M$ as a $\bZ$-module. Indeed, the two $E_8$ factors are spanned by $\{E_0,E_1,\ldots,E_7\}$ and $\{E_{11},E_{12},\ldots,E_{18}\}$, and the generators of $H$ are
\begin{align*}
&3E_0 + 2E_1 + 4E_2+ 6E_3 + 5E_4+4E_5+3E_6+2E_7+E_8,\\
&3E_0 + 2E_1 + 4E_2+ 6E_3 + 5E_4+4E_5+3E_6+2E_7+E_8 + E_9.\\
\end{align*}
Using this fact, we will often express elements of $M$ and/or classes in $\NS(X)$ as combinations of the classes $E_i$. For ease of doing this, we introduce the following notation.

\begin{notation} \label{not:divisors} We will denote the class $\sum_{i=0}^{18}a_iE_i$, for $a_i \in \bZ$, by the notation
\[(a_0;a_1,a_2,a_3,a_4,a_5,a_6,a_7,a_8,a_9,a_{10},a_{11},a_{12},a_{13},a_{14},a_{15},a_{16},a_{17};a_{18}).\]
We separate the end terms $a_0$ and $a_{18}$ by semicolons to indicate that these correspond to the branches off the main chain of the graph.
\end{notation}

\subsection{Geometry of $M$-polarised K3 surfaces}

In this subsection we will outline some of the basic geometric properties of $M$-polarised K3 surfaces, based mostly on ideas from \cite{milpk3s}. 

A generic $M$-polarised K3 surface admits two elliptic fibrations, which Clingher and Doran \cite{milpk3s} call the \emph{standard} and \emph{alternate} fibration. The standard fibration has two singular fibres of Kodaira type $\mathrm{II}^*$ and four of type $\mathrm{I}_1$, plus a unique section. In terms of the divisors $E_i$, the $\mathrm{II}^*$ fibres are given by
\begin{align*}
&(3;2,4,6,5,4,3,2,1,0,0,0,0,0,0,0,0,0;0),\\
&(0;0,0,0,0,0,0,0,0,0,1,2,3,4,5,6,4,2;3),
\end{align*}
and the section is $E_9$. The linear equivalence of these two fibres gives rise to the unique relation between the classes $E_i$:
\begin{equation}\label{eq:E8relation} (3;2,4,6,5,4,3,2,1,0,-1,-2,-3,-4,-5,-6,-4,-2;-3) \sim 0.\end{equation}

The alternate fibration has one singular fibre of Kodaira type $\mathrm{I}_{12}^*$ and six of type $\mathrm{I}_1$, plus two sections. In terms of the divisors $E_i$, the $\mathrm{I}_{12}^*$ fibre is given by
\[(1;0,1,2,2,2,2,2,2,2,2,2,2,2,2,2,1,0;1)\]
and the two sections are $E_1$ and $E_{17}$.

In addition to the fibration structures, a generic $M$-polarised K3 surface also admits a special involution.

\begin{lemma}\label{lem:involution}
Let $X$ be a generic $M$-polarised K3 surface. Then the group of automorphisms which fix the $M$-polarisation has order $2$ and is generated by the fibrewise elliptic involution with respect to the standard fibration and its unique section $E_9$.
\end{lemma}
\begin{proof}
Firstly, we note that the fibrewise involution with respect to the standard fibration fixes the polarisation, so every $M$-polarised K3 surface $X$ admits such an automorphism. We will prove that, if $X$ is generic, then there are no others.

By \cite[Corollary 3.3]{fagkk3s} the group of automorphisms of $X$ that fix the polarisation is cyclic, and by \cite[Theorem 4.2]{aak3satpg} its order divides $12$. Let $\sigma$ be a generator. As $\sigma$ fixes the $M$-polarisation, it maps the curve $E_9$ to itself, and fixes its points of intersection with $E_8$ and $E_{10}$. 

Moreover, $\sigma$ preserves the class of a fibre in the standard fibration, so must act to permute the fibres of this fibration; in particular, it must permute the $\mathrm{I}_1$ fibres. Identifying $E_9$ with $\bP^1$ and setting the points of intersection with $E_8$ and $E_{10}$ to be $0$ and $\infty$, \cite[Lemma 1]{desk3sni} shows that the $\mathrm{I}_1$ fibres occur over $\alpha$, $\alpha^{-1}$, $\beta$, and $\beta^{-1}$, for some $\alpha, \beta \in \bC$. It follows that $\sigma$ acts on $E_9$ as an automorphism of $\bP^1$ which fixes $0$ and $\infty$ and permutes these four points. But if $\alpha$ and $\beta$ are generic the only such automorphism is the identity. So $\sigma$ acts as the identity on $E_9$ and therefore acts fibrewise on the standard fibration.

Finally, as the fibres of the standard fibration do not have $j$-invariant identically $0$ or $1$, the only nontrivial fibrewise automorphism fixing the section $E_9$ is the elliptic involution, so $\sigma$ must be this involution.
\end{proof}

\begin{remark} The fibrewise elliptic involution with respect to the alternate fibration and either of its two sections also fixes the $M$-polarisation. It follows from Lemma \ref{lem:involution} that the fibrewise elliptic involutions with respect to the two fibrations are in fact the same involution. Consequently, for simplicity, we will henceforth refer to this involution simply as the \emph{fibrewise elliptic involution}.
\end{remark}

Using these ideas we define two additional curves on $X$. Let $S$ denote a smooth fibre of the alternate fibration on $X$ and let $T$ denote the $2$-torsion locus of the standard fibration with respect to the unique section $E_9$. As $X$ is very general, $T$ is a smooth curve, which forms a trisection of the standard fibration and a bisection of the alternate fibration. By construction $T$ is part of the fixed locus of the fibrewise elliptic involution; this fixed locus consists of precisely those curves from the set
\[\{T, E_1,E_3,E_5,E_7,E_9,E_{11},E_{13},E_{15},E_{17}\}.\]  
The curves $S$ and $T$ satisfy $S^2 = 0$, $T^2 = 2$, and $S.T = 2$. We can add them into our dual graph as in Figure \ref{fig:augmenteddualgraph}.

\begin{figure}
\begin{tikzpicture}[scale=0.7,thick]
\draw (2,0)--(2,1);
\draw (0,1)--(16,1);
\draw(14,0)--(14,1);
\draw (2,0)--(14,0);
\draw (0,1)--(0,-1)--(16,-1)--(16,1);
\draw (7.95,0)--(7.95,-1);
\draw (8.05,0)--(8.05,-1);

\node[above] at (0,1) {$E_1$};
\node[above] at (1,1) {$E_2$};
\node[above] at (2,1) {$E_3$};
\node[above] at (3,1) {$E_4$};
\node[above] at (4,1) {$E_5$};
\node[above] at (5,1) {$E_6$};
\node[above] at (6,1) {$E_7$};
\node[above] at (7,1) {$E_8$};
\node[above] at (8,1) {$E_9$};
\node[above] at (9,1) {$E_{10}$};
\node[above] at (10,1) {$E_{11}$};
\node[above] at (11,1) {$E_{12}$};
\node[above] at (12,1) {$E_{13}$};
\node[above] at (13,1) {$E_{14}$};
\node[above] at (14,1) {$E_{15}$};
\node[above] at (15,1) {$E_{16}$};
\node[above] at (16,1) {$E_{17}$};
\node[above] at (8,0) {$T$};
\node[below] at (8,-1) {$S$};

\node[left] at (2,0) {$E_0$};
\node[right] at (14,0) {$E_{18}$};

\filldraw[black]
(2,0) circle (3pt)
(14,0) circle (3pt)
(1,1) circle (3pt)
(0,1) circle (3pt)
(2,1) circle (3pt)
(3,1) circle (3pt)
(4,1) circle (3pt)
(5,1) circle (3pt)
(6,1) circle (3pt)
(7,1) circle (3pt)
(8,1) circle (3pt)
(9,1) circle (3pt)
(10,1) circle (3pt)
(11,1) circle (3pt)
(12,1) circle (3pt)
(13,1) circle (3pt)
(14,1) circle (3pt)
(15,1) circle (3pt)
(16,1) circle (3pt)
(8,0) circle (3pt)
(8,-1) circle (3pt)
;

\end{tikzpicture}
\caption{Dual graph of the $(-2)$-curves on $X$ along with the curves $S$ and $T$.}
\label{fig:augmenteddualgraph}
\end{figure}

We can augment our notation for divisors on $X$ with the classes of $S$ and $T$ as follows.

\begin{notation} \label{not:augmenteddivisors} We will denote the $\bQ$-divisor $\sum_{i=0}^{18}a_iE_i + sS + tT$, for $a_i,s,t \in \bQ$, by the notation
\[(a_0;a_1,a_2,a_3,a_4,a_5,a_6,a_7,a_8,a_9,a_{10},a_{11},a_{12},a_{13},a_{14},a_{15},a_{16},a_{17};a_{18}\mid s,t).\]
As divisors on smooth surfaces correspond to sections of line bundles, when there is no potential for confusion we will use the same notation (with $\bZ$-coefficients) for sections of appropriate line bundles.
\end{notation}

The addition of $S$ and $T$ gives rise to two extra relations between our classes:
\begin{align}
\label{eq:Srelation}  (1;0,1,2,2,2,2,2,2,2,2,2,2,2,2,2,1,0;1\mid -1,0) & \sim 0,\\
\label{eq:Trelation} (1;1,2,3,3,3,3,3,3,3,3,3,3,3,3,3,2,1;1\mid 0,-1) & \sim 0.
\end{align}

\subsection{Moduli of $M$-polarised K3 surfaces}\label{sec:moduli}

Based on earlier work of Inose \cite{desk3sni}, Clingher, Doran, Lewis, and Whitcher \cite{milpk3s,nfk3smmp} define a projective model for $M$-polarised K3 surfaces by resolving the $E_6$ and $A_{11}$ singularities in the singular quartic surfaces
\begin{equation}\label{eq:CDform}\{y^2zw - 4x^3z+3axzw^2+bzw^3 - \tfrac{1}{2}(dz^2w^2+w^4) = 0\} \subset \bP^3[w,x,y,z],\end{equation}
for parameters $a,b,d \in \bC$ with $d \neq 0$. The surfaces in this family are smooth away from the locus
\[\{a^6 +b^4+d^2-2a^3b^2-2a^3d-2b^2d = 0\}.\]
Along this locus the surfaces acquire an additional $A_1$ singularity; if one resolves this then the lattice polarisation jumps to $H \oplus E_8 \oplus E_8 \oplus A_1$.

\cite[Theorem 3.2]{nfk3smmp} shows that two such quartics, with defining parameters $(a_1,b_1,d_1)$ and $(a_2,b_2,d_2)$, determine isomorphic $M$-polarised K3 surfaces if and only if 
\[(a_1,b_1,d_1) = (\lambda^2 a_2, \lambda^3 b_2, \lambda^6 d_2),\] 
from which it follows that the affine variety
\[\{(a,b,d) \in \bW\bP(2,3,6) : d \neq 0\}\]
forms a coarse moduli space for $M$-polarised K3 surfaces. This coarse moduli space can be compactified to the weighted projective space $\bW\bP(2,3,6)$.

The Baily-Borel compactification for the moduli space of $M$-polarised K3 surfaces contains a unique $1$-cusp, which can be identified in the compactification above as the locus
\[\{(a,b,0) \in \bW\bP(2,3,6) : a^3 \neq b^2\},\]
and a unique $0$-cusp, which can be identified as the point $(a,b,d)=(1,1,0)$.

\section{Tyurin degenerations}\label{sec:Tyurin}

\subsection{Background}

In this subsection we summarise some basic information about degenerations of K3 surfaces; more details may be found in \cite{bgd1,cmsak3s,cmk3s}. 

By a \emph{degeneration} of K3 surfaces, we mean a proper, flat, surjective morphism $\pi \colon \calX \to \Delta$ from a K\"{a}hler manifold to the open complex unit disc, such that the restriction to the punctured disc is a smooth morphism with all fibres K3 surfaces. We say that a degeneration of K3 surfaces is \emph{projective} if $\pi$ is a projective morphism. We denote the fibre over a point $p \in \Delta$ by $X_p$. A degeneration of K3 surfaces is called \emph{semistable} if the central fibre $X_0$ is a simple normal crossings divisor, and is called a \emph{Kulikov model} if it is semistable and $\omega_{\calX} \cong \calO_{\calX}$.

Kulikov degenerations can be classified into Types I, II, and III. Degenerations of Type II (resp. III) occur over $1$-cusps (resp. $0$-cusps) in the Baily-Borel compactification of the appropriate moduli space of K3 surfaces. In this paper we are interested in a special kind of Type II degeneration, known as a Tyurin degeneration.

\begin{definition}
A semistable degeneration of K3 surfaces is called a \emph{Tyurin degeneration} if its central fibre $X_0$ consists of two rational surfaces $Y_1$ and $Y_2$ glued along a smooth anticanonical elliptic curve $D$.
\end{definition}

By adjunction, Tyurin degenerations are automatically Kulikov models. However, they are also not unique: given a smooth family of K3 surfaces over the punctured disc, there may be multiple different ways to complete it to a Tyurin degeneration. Such non-uniqueness is common for Tyurin degenerations in the analytic category, where one can always analytically flop $(-1)$-curves from $Y_1$ to $Y_2$ (or vice versa) to obtain different Tyurin degenerations, but is less well-understood in the projective setting.

In this paper we are interested in studying projective Tyurin degenerations. We wish to categorise the possible projective Tyurin degenerations that can occur over the $1$-cusp in the moduli space of $M$-polarised K3 surfaces.

In order to do this, we need to understand how to place a lattice polarisation on a degeneration. The following definition is adapted from \cite[Definition 2.1]{flpk3sm}.

\begin{definition} \label{def:Mpoldegeneration}  A degeneration of K3 surfaces $\pi \colon \calX \to \Delta$ is called \emph{$M$-polarised} if there is a trivial local subsystem $\calM$ of $R^2\pi_*\bZ$ over the punctured disc $\Delta^*$, with fibre $M$, so that for each $p \in \Delta^*$ the fibre $\calM_p \subset H^2(X_p,\bZ)$ of $\calM$ over $p$ defines an $M$-polarisation on $X_p$.
\end{definition}

\begin{remark} Triviality of the local subsystem $\calM$ implies that the very irrational vector $h \in M \otimes \bR$ is monodromy invariant. The $M$-polarisation condition then includes the assumption that $h \in \NS(X_p) \otimes \bR$ must be ample for all $p \in \Delta^*$.
\end{remark}

\subsection{Analytic Tyurin degenerations}\label{sec:analytictyurin}

We begin our study by classifying $M$-polarised Tyurin degenerations without assuming that the morphism $\pi$ is projective. In order to do this, we start by studying the rational surfaces that will appear in such degenerations.

\begin{definition} \label{def:Vd} Let $V_d$, with $d \geq 0$, denote the rational surface obtained from $\bP^2$ by blowing up $d$ times at an inflection point $p$ of a smooth cubic curve $C$. Let $D$ denote the strict transform of $C$. Let $F_i$, for $1 \leq i \leq d$, denote the strict transform of the exceptional curve arising from the $i$th blow-up, and let $F_0$ denote the strict transform of the inflectional tangent line at $p$. Note that $F_i^2 = -2$ for $1 \leq i \leq d-1$ and $F_d^2 = -1$. For $d \geq 3$ we also have $F_0^2 = -2$; in this case the dual graph of these curves is shown in Figure \ref{fig:blowupdualgraph}.

Let $\overline{V}_1$ denote the Hirzebruch surface $\bF_2$ and let $D \subset \overline{V}_1$ be a smooth anticanonical curve. Let $F_1$ denote the unique $(-2)$-curve in $\overline{V}_1$ and let $F_2$ denote a fibre of the ruling on $\bF_2$ which lies tangent to $D$ at a point $p$. Note that $F_2^2 = 0$.
\end{definition}

\begin{figure}
\begin{tikzpicture}[scale=0.7,thick]
\draw (2,0)--(2,1);
\draw (0,1)--(3,1);
\draw[dashed] (3,1)--(4.5,1);
\draw (4.5,1)--(6.5,1);

\node[above] at (0,1) {$F_1$};
\node[above] at (1,1) {$F_2$};
\node[above] at (2,1) {$F_3$};
\node[above] at (3,1) {$F_4$};
\node[above] at (4.5,1) {$F_{d-1}$};
\node[above] at (5.5,1) {$F_d$};
\node[above] at (6.5,1) {$D$};

\node[left] at (2,0) {$F_0$};

\filldraw[black]
(2,0) circle (3pt)
(1,1) circle (3pt)
(0,1) circle (3pt)
(2,1) circle (3pt)
(3,1) circle (3pt)
(4.5,1) circle (3pt)
(5.5,1) circle (3pt)
(6.5,1) circle (3pt)
;

\end{tikzpicture}
\caption{Dual graph of curves on $V_d$ for $d \geq 3$.}
\label{fig:blowupdualgraph}
\end{figure}

Note that for each surface $V_d$ and $\overline{V}_d$, the curve $D$ is smooth and anticanonical with $D^2 = 9-d$.  In a mild abuse of notation, we let $p \in D$ denote the strict transform of the point $p \in C$.

In $V_1 \cong \bF_1$ and $\overline{V}_1\cong \bF_2$ the only irreducible curve of negative self-intersection is $F_1$, and in $V_2$ the only such curves are $F_0$, $F_1$, and $F_2$. For $d \geq 3$ we have the following result.

\begin{lemma} \label{lem:Vdcurves} Let $d \geq 3$. Then $F_d$ is the only rational $(-1)$-curve on $V_d$ and $\{F_0,\ldots,F_{d-1}\}$ are the only rational $(-2)$-curves on $V_d$.
\end{lemma}

\begin{proof} This proof is based on ideas of Looijenga \cite{rsac} and Harbourne \cite{bup2bd}; we refer the reader to those papers for more details. Suppose that $V$ is a surface obtained by blowing-up $\bP^2$ a total of $d$ times. Then $V$ contains an \emph{exceptional configuration}: an ordered set $\{\calE_0,\calE_1,\ldots,\calE_d\}$, where $\calE_0$ is the total transform of a line in $\bP^2$ and $\calE_i$, for $1 \leq i \leq d$, is the total transform of the exceptional $(-1)$-curve arising from the $i$th blow-up. Exceptional configurations on $V$ are in bijective correspondence with birational morphisms $V \to \bP^2$.

Associated to an exceptional configuration is a set of simple roots $r_0 = \calE_0 - \calE_1 - \calE_2 - \calE_3$ and $r_i = \calE_i - \calE_{i+1}$ for $1 \leq i \leq d-1$. Each of these roots defines a fundamental reflection $s_i(x) = x + \langle x, r_i \rangle r_i$, for $x \in \NS(V) \cong \Pic(V)$, and the fundamental reflections generate a Weyl group $W$ which acts by isometries on $\NS(S)$. \cite[Theorem 0.1]{bup2bd} then shows that all exceptional configurations on $V$ are $W$-translates of one another.

We apply this to our setting as follows. Let $V = V_d$ for some $d \geq 3$ and let 
\begin{align*}
\calE_0 &= [F_0] + [F_1] + 2[F_2] + 3[F_3] + 3[F_4] + \cdots + 3[F_d],\\ 
\calE_i &= [F_d] + [F_{d-1}] + \cdots + [F_i] \quad \text{for each } 1 \leq i \leq d,
\end{align*}
where $[F_i] \in \NS(V_d)$ denotes the class of the curve $F_i$. Then $\calE := \{\calE_0,\calE_1,\ldots,\calE_d\}$ is an exceptional configuration for $V_d$ and the simple roots $r_i = [F_i]$ for all $0 \leq i \leq d-1$. To prove the lemma, it suffices to show that no other exceptional configurations exist on $V_d$. To do this, we show that no nontrivial $W$-translate of $\calE$ can give an exceptional configuration.

Define $C \subset \NS(V_d) \otimes \bR$ by 
\[C := \{x \in \NS(V_d) \otimes \bR : \langle x, r_i \rangle > 0 \text{ for all } 0 \leq i \leq d-1\}.\]
Then $C$ is a fundamental chamber for the $W$-action. Let $\overline{C}$ denote its closure.

Now let $w \in W$ be any element. If $w(\calE_0) \notin \overline{C}$ then $\langle w(\calE_0),r_i \rangle < 0$ for some $r_i$ and, since the $r_i = [F_i]$ are classes of effective curves, we find that $w(\calE_0)$ is not nef. Then by \cite[Theorem  1.1]{bup2bd}, $w(\calE)$ cannot be an exceptional configuration.

We may therefore assume that $w(\calE_0) \in \overline{C}$. But then, by \cite[(3.2)]{rsac}, we have that $w(\calE_0) = \calE_0$ and $w$ is in the subgroup of $W$ generated by the $s_i$ with $1 \leq i \leq d-1$. An easy computation shows that if $1 \leq i \leq d-1$, then the reflection $s_i$ exchanges $\calE_i$ and $\calE_{i+1}$ and fixes the other $\calE_j$. So $w$ must fix $\calE_0$ and permute the $\calE_i$ for $1 \leq i \leq d$.

If $w$ is nontrivial, then there must exist $i > j > 0$ such that $w(\calE_i) = \calE_{i'}$ and $w(\calE_j) = \calE_{j'}$ with $j'>i'>0$. Then
\[ w(\calE_i) - w(\calE_j) = \calE_{i'} - \calE_{j'} = r_{i'} + r_{i'+1} + \cdots + r_{j'-1}\]
and the right-hand side is the class of the effective divisor $F_{i'}+ \cdots + F_{j'-1}$. But then, by \cite[Theorem  1.1]{bup2bd}, $w(\calE)$ cannot be an exceptional configuration.

Consequently, $w(\calE)$ is only an exceptional configuration if $w$ is trivial. Therefore, by \cite[Theorem 0.1]{bup2bd}, $\calE$ is the only exceptional configuration on $V_d$.
\end{proof}

The surfaces $V_d$ and $\overline{V}_d$ will be the components of the central fibres of our Tyurin degenerations. The following definition describes these central fibres.

\begin{definition} The \emph{stable surface $(V_d,V_{18-d})$} (resp. $(\overline{V}_1,V_{17})$), for $d \leq 9$, is the normal surface obtained by gluing the surfaces $V_d$ and $V_{18-d}$ (resp. $\overline{V}_1$ and $V_{17}$) from Definition \ref{def:Vd} via an isomorphism of the anticanonical curves $D$ that identifies the points $p$. Let $F_i$ denote the special curves in $V_d$ (resp. $\overline{V}_1$) from Definition \ref{def:Vd} and let $F_i'$ denote the special curves in $V_{18-d}$.
\end{definition}

The following is the main theorem of this section, which provides a description of the possible $M$-polarised Tyurin degenerations. Note that this result does not assume that $\pi$ is projective.

\begin{theorem} \label{thm:analyticclassification} Let  $\pi \colon \calX \to \Delta$ be an $M$-polarised Tyurin degeneration of K3 surfaces. Then the central fibre $X_0$ is a stable surface $(V_d,V_{18-d})$ or $(\overline{V}_1,V_{17})$. Moreover, up to labelling, the generators $E_i$ of the lattice $M$ degenerate to curves on $X_0$ as follows.
\begin{itemize}
\item If $3 \leq d \leq 9$, the curves $E_i$ degenerate to $F_i$ if $i < d$, $F_{18-i}'$ if $i > d$, and $(F_d+F_{18-d}')$ if $i=d$; the case $d=6$ is illustrated by Figure \ref{fig:tyurindualgraph}
\item In $(V_2,V_{16})$, the curve $E_0$ degenerates to $(F_0 + F_{16}')$, the curve $E_1$ degenerates to $F_1$, the curve $E_2$ degenerates to $(F_2 + F_{16}')$, and the remaining $E_i$ degenerate to $F_{18-i}'$.
\item In $(V_1,V_{17})$, the curve $E_0$ degenerates to $(F_0 + 2F_{17}'+F_{16}')$, the curve $E_1$ degenerates to $(F_1+F_{17}')$, and the remaining $E_i$ degenerate to $F_{18-i}'$.
\item In $(V_0,V_{18})$, the curve $E_0$ degenerates to $(F_0 + 3F_{18}' + 2F_{17}'+F_{16}')$ and the remaining $E_i$ degenerate to $F_{18-i}'$.
\item In $(\overline{V}_1,V_{17})$, the curve $E_0$ degenerates to $F_{16}'$, the curve $E_1$ degenerates to $F_1$, the curve $E_2$ degenerates to $(F_2 + 2F_{17}' + F_{16}')$, and the remaining $E_i$ degenerate to $F_{18-i}'$.
\end{itemize}
\end{theorem}

\begin{figure}
\begin{tikzpicture}[scale=0.7,thick]
\draw (2,2)--(2,3);
\draw (0,3)--(5,3);
\draw (5,1)--(16,1);
\draw (14,0)--(14,1);
\draw[dashed] (5,3)--(5,1);

\node[above] at (0,3) {$F_1$};
\node[above] at (1,3) {$F_2$};
\node[above] at (2,3) {$F_3$};
\node[above] at (3,3) {$F_4$};
\node[above] at (4,3) {$F_5$};
\node[above] at (5,3) {$F_6$};
\node[above left] at (5,1) {$F_{12}'$};
\node[above] at (6,1) {$F_{11}'$};
\node[above] at (7,1) {$F_{10}'$};
\node[above] at (8,1) {$F_{9}'$};
\node[above] at (9,1) {$F_8'$};
\node[above] at (10,1) {$F_7'$};
\node[above] at (11,1) {$F_6'$};
\node[above] at (12,1) {$F_5'$};
\node[above] at (13,1) {$F_4'$};
\node[above] at (14,1) {$F_3'$};
\node[above] at (15,1) {$F_2'$};
\node[above] at (16,1) {$F_1'$};

\node[left] at (2,2) {$F_0$};
\node[right] at (14,0) {$F_0'$};

\filldraw[black]
(2,2) circle (3pt)
(14,0) circle (3pt)
(1,3) circle (3pt)
(0,3) circle (3pt)
(2,3) circle (3pt)
(3,3) circle (3pt)
(4,3) circle (3pt)
(5,3) circle (3pt)
(5,1) circle (3pt)
(6,1) circle (3pt)
(7,1) circle (3pt)
(8,1) circle (3pt)
(9,1) circle (3pt)
(10,1) circle (3pt)
(11,1) circle (3pt)
(12,1) circle (3pt)
(13,1) circle (3pt)
(14,1) circle (3pt)
(15,1) circle (3pt)
(16,1) circle (3pt)
;

\end{tikzpicture}
\caption{Limits of curves $E_i$ on $(V_6,V_{12})$. Curves on $V_6$ are to the left of the dashed edge and curves on $V_{12}$ are to the right. The dashed edge denotes the point $p \in D$. The curves $E_0,\ldots,E_5$ degenerate to $F_0,\ldots,F_5$, the curves $E_{7}\ldots,E_{18}$ degenerate to $F_{11}',\ldots,F_{0}'$, and the curve $E_6$ degenerates to $(F_6+F_{12}')$.}
\label{fig:tyurindualgraph}
\end{figure}

\begin{proof} We begin by recalling some of the lattice theory of Tyurin degenerations; see \cite[Section 3]{npgttk3s} and \cite[Section 2.2]{cmsak3s} for details. Given a Tyurin degeneration $\pi\colon \calX \to \Delta$ of K3 surfaces with central fibre $X_0 = Y_1 \cup_D Y_2$, the weight filtration in the limiting mixed Hodge structure $H^2_{\lim}(X_p)$ is defined over $\bZ$ and is given by
\[0 \subset I \subset I^{\perp} \subset H^2(X_p,\bZ),\]
where $I \subset H^2(X_p,\bZ)$ is a rank $2$ isotropic sublattice, the choice of which is unique up to isometries of $H^2(X_p,\bZ) \cong H\oplus H \oplus H \oplus E_8 \oplus E_8$. From the Clemens-Schmid sequence we obtain an exact sequence of lattices (see \cite[Lemma 3.5]{npgttk3s})
\[0 \longrightarrow \bZ \xi \longrightarrow \xi^{\perp} \longrightarrow I^{\perp}/I \longrightarrow 0,\]
where $\xi \in H^2(Y_1,\bZ) \oplus H^2(Y_2,\bZ)$ is the class of $(\calO_{Y_1}(D),\calO_{Y_2}(-D))$. As $I$ is unique up to isometry, we can take $I$ to be generated by primitive isotropic classes in the first two copies of $H$ in $H\oplus H \oplus H \oplus E_8 \oplus E_8$, from which it follows that $I^{\perp}/I \cong M$.

There is a map $\widetilde{\psi} \in \Hom(\xi^{\perp},D)$ defined by
\[\widetilde{\psi}(\calL_1,\calL_2) = \calL_1|_{D} \otimes \calL_2^{-1}|_D \in \Pic^0(D) \cong D,\]
where $(\calL_1,\calL_2) \in \xi^{\perp} \subset H^2(Y_1,\bZ) \oplus H^2(Y_2,\bZ) \cong \Pic(Y_1) \oplus \Pic(Y_2)$, and $\widetilde{\psi}$ necessarily satisfies $\widetilde{\psi}(\xi) = 1$. This map therefore descends to $\psi \in \Hom(I^{\perp}/I, D)$; this $\psi$ is called the \emph{limiting period point} of the Tyurin degeneration.

Alexeev and Engel \cite[Section 4]{cmk3s} have shown how to extend this to the lattice polarised setting. Given an $M$-polarised Tyurin degeneration of K3 surfaces, one obtains a lattice embedding $M \hookrightarrow I^{\perp}/I$ and, moreover, the limiting period point $\psi$ must satisfy $\psi(\alpha) = 1$ for all $\alpha \in M$. Since $I^{\perp}/I \cong M$, this embedding is in fact an isometry, so the limiting period point must satisfy $\psi(\alpha) = 1$ for all $\alpha \in I^{\perp}/I$. This is equivalent to  $\widetilde{\psi}(\alpha) = 1$ for all $\alpha \in \xi^{\perp}$.

By the classification of rational surfaces with smooth anticanonical divisor, there are two possibilities for the surfaces $Y_i$. For ease of notation we set $i=1$ in this analysis; the argument for $i=2$ is analogous.
\begin{itemize}
\item $Y_1$ is a blow-up of $\bP^2$ in $d \geq 0$ points lying on a smooth cubic curve $C$, and $D$ is the strict transform of $C$. In this case, let $L \subset Y_1$ denote the pull-back of a line from $\bP^2$ and let $E_1, E_2 \subset Y_1$ be any two $(-1)$-curves. Then $\widetilde{\psi}(\calO_{Y_1}(E_1-E_2),\calO_{Y_2}) = \widetilde{\psi}(\calO_{Y_1}(L-3E_1),\calO_{Y_2}) = 1$, so $E_1$ and $E_2$ intersect $D$ in the same point $p$, which must be an inflection point of the cubic $C$. It follows that $Y_1 \cong V_d$.
\item $Y_1$ is either $\bP^1 \times \bP^1$ or $\bF_2$ and $D$ is a smooth anticanonical curve with $D^2 = 8$.
\end{itemize}
Moreover, by semistability we must have $(D|_{Y_1})^2 = -(D|_{Y_2})^2$. Relabelling if necessary, we may assume that $(D|_{Y_2})^2 \leq 0$. By the analysis above, this implies that $Y_2 \cong V_{18-d}$ for $0 \leq d \leq 9$. Let $E \subset Y_2$ be the $(-1)$-curve (which is unique by Lemma \ref{lem:Vdcurves}). Then:
\begin{itemize}
\item If $Y_1 \cong V_d$ for $d \geq 1$, let $F \subset Y_1$ be the unique $(-1)$-curve.  Then $\widetilde{\psi}(\calO_{Y_1}(F),\calO_{Y_2}(E)) = 1$, so $E$ and $F$ necessarily intersect the same point $p \in D$.
\item If $Y_1 \cong V_0 \cong \bP^2$, let $L \subset Y_1$ be a line. Then $\widetilde{\psi}(\calO_{Y_1}(L),\calO_{Y_2}(3E)) = 1$, so $p = E \cap D$ is necessarily identified with an inflection point on $D \subset \bP^2$.
\item If $Y_1 \cong \bF_2$, let $F \subset \bF_2$ be a fibre of the ruling. Then $\widetilde{\psi}(\calO_{Y_1}(F),\calO_{Y_2}(2E)) = 1$, so $p = E \cap D$ is necessarily identified with a point where $D$ lies tangent to a fibre of the ruling.
\item If $Y_1 \cong \bP^1 \times \bP^1$, let $F_1,F_2 \subset \bP^1 \times \bP^1$ be fibres of the two rulings. Then $\widetilde{\psi}(\calO_{Y_1}(F_i),\calO_{Y_2}(2E)) = 1$ for each $i \in \{1,2\}$, so $p = E \cap D$ is necessarily identified with a point where $D$ lies tangent to fibres of both rulings. There are no such points for $D$ smooth and anticanonical, so this case is impossible. 
\end{itemize}
It follows that $X_0$ must be a stable surface $(V_d,V_{18-d})$ or $(\overline{V}_1,V_{17})$. Finally, the statement about the limits of the classes $E_i$ follows from the fact that these limits must be effective divisors of self-intersection $-2$, along with the description of negative curves in the surfaces $V_d$ from Lemma \ref{lem:Vdcurves}.
\end{proof}

\begin{remark} Note that all of the degenerations with central fibre $(V_d,V_{18-d})$, with $d \geq 2$, are related by a single chain of flops. In $V_2$ there is a choice of two $(-1)$-curves, $F_0$ and $F_2$, that one may flop to proceed further: flopping $F_0$ gives $(\overline{V}_1,V_{17})$, at which point the chain terminates, and flopping $F_2$ gives $({V}_1,V_{17})$ followed by $({V}_0,V_{18})$. It follows from Lemma \ref{lem:Vdcurves} that no other flops are possible.
\end{remark}

In the remainder of this paper, we will show that all of the central fibres described in Theorem \ref{thm:analyticclassification} can be realised by projective $M$-polarised Tyurin degenerations. To do this we begin by constructing a \emph{projective model} for $M$-polarised K3 surfaces, then perform an appropriate degeneration in the ambient projective space. We use two approaches to construct these projective models, which we will now illustrate by giving two derivations of the projective normal form \eqref{eq:CDform}.

\subsection{Linear system approach}\label{sec:linearsystem}

As usual, we let $X$ denote a very general $M$-polarised K3 surface. In this approach we begin with a nef and big $\bQ$-divisor $D$ on $X$, then exhibit sections that generate the canonical ring 
\[R(X,D) = \bigoplus_{n \geq 0} H^0(X, \calO_X(\lfloor nD \rfloor)).\]
It follows from \cite[Corollary 5]{fk3s} that if $kD$ is a Cartier divisor, then $R(X,D)$ is generated in degrees $\leq 3k$ and we have a birational morphism $X \to \Proj R(X,D)$ to the canonical model of $(X,D)$. In the cases studied, a choice of homogeneous generators for $\Proj R(X,D)$ will realize  $\Proj R(X,D)$ as a variety in weighted projective space. We then explicitly compute the image of $X$ inside this weighted projective space to obtain our projective model.

For example, using Notation \ref{not:augmenteddivisors}, we could consider the divisor 
\[D=(4;3,6,9,8,7,6,5,4,3,2,1,0,0,0,0,0,0;0\mid 0,0).\]
Then $D$ is a nef and big Cartier divisor with $D^2 = 4$, and by Riemann-Roch we have $h^0(X, \calO_X(D)) = 4$. Using relations \eqref{eq:E8relation}, \eqref{eq:Srelation}, and \eqref{eq:Trelation}, we find three independent divisors that are linearly equivalent to $D$:
\begin{align*}
&(1;1,2,3,3,3,3,3,3,3,3,3,3,4,5,6,4,2;3\mid 0,0),\\
&(0;1,1,1,1,1,1,1,1,1,1,1,1,2,3,4,3,2;2\mid 1,0),\\
&(0;0,0,0,0,0,0,0,0,0,0,0,0,1,2,3,2,1;2\mid 0,1).
\end{align*}
These four divisors form a base-point free linear system. Denote the corresponding sections of $\calO_X(D)$ by $z,w,x,y$, respectively. Then $w,x,y,z$ generate the graded ring $R(X,D)$, so we obtain a birational morphism $\varphi\colon X \to \bP^3[w,x,y,z]$ whose image is a quartic hypersurface.

From the intersection properties of the divisors above, one can see that $\varphi$ contracts $E_1,\ldots,E_{11}$ to a singularity of type $A_{11}$ and $E_{13},\ldots,E_{18}$ to a singularity of type $E_6$. The curve $E_0$ is mapped to the line $z=w=0$ and the curve $E_{12}$ is mapped to the line $x=w=0$. 

Let $f(w,x,y,z)$ be the defining equation of $\varphi(X)$. From the description of the divisors above we have the following.
\begin{itemize}
\item $f(0,x,y,z)$ is a nonzero multiple of $x^3z$.
\item $f(w,0,y,z)$ may be written as $wf_3(w,y,z)$, for some cubic $f_3$ defining the image of $S$ in the hyperplane $\{x=0\}$. Moreover, the cubic $f_3$ intersects the line $w=0$ in the two points $w=z=0$ (with multiplicity $1$) and $w=y=0$ (with multiplicity $2$), so is a nonzero multiple of $zy^2+w(\text{conic})$.
\item $f(w,x,0,z)$ is a quartic curve defining the image of $T$ in the hyperplane $\{y=0\}$.
\item $f(w,x,y,0)$ is a nonzero multiple of $w^4$, this multiple can be set to $1$ by global rescaling.
\end{itemize}
This gives a candidate form for $f_4(w,x,y,z)$:
\[w^4 +a_1x^3z + wz(a_2w^2 + a_3wx + a_4wy + a_5wz + a_6x^2 + a_7xy + a_8xz + a_{9}y^2),\]
where $a_i \in \bC$ and $a_1,a_{9}$ are nonzero.

We can place further restrictions on this form. The quotient of $X$ by the fibrewise elliptic involution is a surface containing a configuration of curves as shown in Figure \ref{fig:quotientdivisors}, where the numbers on the vertices denote self-intersection numbers. Note that the ramification curve $T$ is isomorphic to its image under this quotient. The birational map contracting the curves $E_{13},\ldots,E_{18}$ descends to this quotient, where it creates a singularity of type $A_2$ in the branch curve. We thus see that the quartic curve $f(w,x,0,z)$ should have a singularity of type $A_2$ at the point $x=w=0$. This forces $a_8 = 0$ in the form above.

\begin{figure}
\begin{tikzpicture}[scale=0.7,thick]
\draw (2,0)--(2,1);
\draw (0,1)--(16,1);
\draw(14,0)--(14,1);
\draw (2,0)--(14,0);
\draw (0,1)--(0,-1)--(16,-1)--(16,1);
\draw (7.95,0)--(7.95,-1);
\draw (8.05,0)--(8.05,-1);

\node[above] at (0,1) {$-4$};
\node[above] at (1,1) {$-1$};
\node[above] at (2,1) {$-4$};
\node[above] at (3,1) {$-1$};
\node[above] at (4,1) {$-4$};
\node[above] at (5,1) {$-1$};
\node[above] at (6,1) {$-4$};
\node[above] at (7,1) {$-1$};
\node[above] at (8,1) {$-4$};
\node[above] at (9,1) {$-1$};
\node[above] at (10,1) {$-4$};
\node[above] at (11,1) {$-1$};
\node[above] at (12,1) {$-4$};
\node[above] at (13,1) {$-1$};
\node[above] at (14,1) {$-4$};
\node[above] at (15,1) {$-1$};
\node[above] at (16,1) {$-4$};
\node[above] at (8,0) {$4$};
\node[below] at (8,-1) {$0$};

\node[left] at (2,0) {$-1$};
\node[right] at (14,0) {$-1$};

\filldraw[black]
(2,0) circle (3pt)
(14,0) circle (3pt)
(1,1) circle (3pt)
(0,1) circle (3pt)
(2,1) circle (3pt)
(3,1) circle (3pt)
(4,1) circle (3pt)
(5,1) circle (3pt)
(6,1) circle (3pt)
(7,1) circle (3pt)
(8,1) circle (3pt)
(9,1) circle (3pt)
(10,1) circle (3pt)
(11,1) circle (3pt)
(12,1) circle (3pt)
(13,1) circle (3pt)
(14,1) circle (3pt)
(15,1) circle (3pt)
(16,1) circle (3pt)
(8,0) circle (3pt)
(8,-1) circle (3pt)
;

\end{tikzpicture}
\caption{Dual graph of curves on the quotient by the fibrewise elliptic involution.}
\label{fig:quotientdivisors}
\end{figure}

As $a_{10} \neq 0$, completing the square in $y$ allows us to eliminate the $wxyz$ and $w^2yz$ terms, and as $a_1 \neq 0$ we can complete the cube in $x$ to eliminate the $x^2wz$ term. We are left with the form
\[f_4(w,x,y,z) = w^4 +b_1x^3z + b_2w^3 z+ b_3w^2xz + b_4w^2z^2 + b_5wy^2z,\]
for some $b_i \in \bC$ with $b_1, b_5$ nonzero. Finally, rescaling variables allows us to obtain precisely the projective normal form \eqref{eq:CDform}.

\subsection{Toric approach}\label{sec:toricapproach}

For this approach we follow the ideas of \cite[Section 3]{tfmsdp}. An $M$-polarised K3 surface is mirror, in the sense of Dolgachev \cite{mslpk3s}, to an $H$-polarised K3 surface. A generic $H$-polarised K3 surface can be expressed as the minimal resolution of a hypersurface of degree $12$ in the weighted projective space $\bW\bP(1,1,4,6)$. By the Batyrev mirror construction \cite{dpmscyhtv}, a generic $M$-polarised K3 surface can therefore be realized torically as an anticanonical hypersurface in $\bW\bP(1,1,4,6)^{\circ}$, the polar dual of $\bW\bP(1,1,4,6)$.

We derive a normal form for such an anticanonical hypersurface following the method of \cite[Section 4.2]{msag}. Let $L := \bZ^3$ denote the standard lattice and let $L_{\bR} := L \otimes \bR$. An anticanonical hypersurface in $\bW\bP(1,1,4,6)$ corresponds to the reflexive polytope in $L_{\bR}$ with vertices
\begin{align*}
u_0 &= (-1,-1,-1), & u_1&= (11,-1,-1),\\
u_2&=(-1,2,-1), & u_3&=(-1,-1,1).
\end{align*}
The rays of the fan $\Sigma^{\circ} \subset L_{\bR}$ of the toric variety $\bW\bP(1,1,4,6)^{\circ}$ are generated by the above set of vectors. These vectors generate a sublattice $L' \subset L$, whose index is given by the determinant of the matrix with rows $u_1$, $u_2$, $u_3$; this determinant is $6$. The quotient $L/L'$ is the group of order $6$ defined by
\[G := \{(d_0,d_1,d_2,d_3) \in \bZ_{12}^2 \oplus \bZ_3 \oplus \bZ_2 : d_0+d_1+4d_2+6d_3 \equiv 0 \pmod {12}\}/\bZ_{12},\]
where $\bZ_{12} \subset \bZ_{12}^2 \oplus \bZ_3 \oplus \bZ_2$ is embedded diagonally and $\bZ_n := \bZ/n\bZ$ denotes integers modulo $n$.

As $u_0 + u_1+u_2+u_3=0$, a basis for $L'$ is given by $\{u_1,u_2,u_3\}$, so if we view $\Sigma^{\circ}$ as a fan in $L'\otimes \bR$, then $\Sigma^{\circ}$ is the standard fan for $\bW\bP(1,1,4,6)$. In the overlattice $L$, there is a $G$-action on $\bW\bP(1,1,4,6)$ given by
\[(x_0,x_1,x_2,x_3) \longmapsto (\omega^{d_0}x_0, \omega^{d_1}x_1,\omega^{4d_2}x_2, \omega^{6d_3}x_3),\]
where $(d_0,d_1,d_2,d_3) \in G$ and $\omega$ is a primitive twelfth root of unity. $\bW\bP(1,1,4,6)^{\circ}$ is the quotient of $\bW\bP(1,1,4,6)$ by this action.

The total coordinate ring of $\bW\bP(1,1,4,6)^{\circ}$ is given by $\bC[x_0,x_1,x_2,x_3]$ graded by the divisor class group $\Cl(\bW\bP(1,1,4,6)^{\circ})$; see \cite[Section 5.2]{tv} for details. To describe the grading explicitly we use the exact sequence
\[0 \longrightarrow L \longrightarrow \bZ^4 \stackrel{\deg}{\longrightarrow} \Cl(\bW\bP(1,1,4,6)^{\circ}) \longrightarrow 0,\]
where the first map is given by the matrix whose rows are $u_0$, $u_1$, $u_2$, $u_3$. It follows that $\Cl(\bW\bP(1,1,4,6)^{\circ}) \cong \bZ \oplus G$ and the degree map is given by
\[\deg(d_0,d_1,d_2,d_3):= \big(d_0+d_1+4d_2+6d_3, (-d_1-4d_2-6d_3,d_1,d_2,d_3)\big) \in \bZ \oplus G.\]
Thus the grading on $\bC[x_0,x_1,x_2,x_3]$ is given by taking the monomial $x_0^{d_0}x_1^{d_1}x_2^{d_2}x_3^{d_3}$ to $\deg(d_0,d_1,d_2,d_3) \in \bZ \oplus G$.

The anticanonical class of $\bW\bP(1,1,4,6)^{\circ}$ has degree $\deg(1,1,1,1) = (12,0)$. A generic polynomial in $\bC[x_0,x_1,x_2,x_3]$ of degree $(12,0)$ is given by
\[a_0x_0^{12} + a_1x_1^{12}+a_2x_2^3+a_3x_3^2+a_4x_0^6x_1^6 + a_5x_0^4x_1^4x_2+a_6x_0^2x_1^2x_2^2+a_7x_0x_1x_2x_3+a_8x_0^3x_1^3x_3.\]
Using the torus action we can rescale
\[x_0 \longmapsto a_0^{-\frac{1}{12}}x_0, \quad x_1 \longmapsto a_0^{\frac{1}{12}}a_2^{\frac{1}{3}}a_3^{\frac{1}{2}}a_7^{-1}x_1, \quad x_2 \longmapsto a_2^{-\frac{1}{3}}x_2, \quad x_3 \longmapsto a_3^{-\frac{1}{2}}x_3.\] 
Our equation becomes
\[x_0^{12} + b_0x_1^{12}+x_2^3+x_3^2+b_1x_0^6x_1^6 + b_2x_0^4x_1^4x_2+b_3x_0^2x_1^2x_2^2+x_0x_1x_2x_3+b_4x_0^3x_1^3x_3,\]
for $b_0 = \frac{a_0a_1a_2^4a_3^6}{a_4^{12}}$, $b_1=\frac{a_2^2a_3^3}{a_4^5}$, $b_2 = \frac{a_2a_3^2a_5}{a_4^4}$, $b_3=\frac{a_3a_6}{a_4^2}$, and $b_4 = \frac{a_2a_3a_8}{a_4^3}$. Assigning weights $(1,1,4,6)$ to the variables $(x_0,x_1,x_2,x_3)$, we see that this equation is weighted homogeneous, so defines a hypersurface in $\bW\bP(1,1,4,6)$. A generic anticanonical hypersurface in $\bW\bP(1,1,4,6)^{\circ}$ is given by the quotient of such a hypersurface by the group $G$.

The quotient $\bW\bP(1,1,4,6)/G$ is realised by the map
\begin{align*}
\psi\colon \bW\bP(1,1,4,6) &\longrightarrow \bP^8\\
(x_0,x_1,x_2,x_3) &\longmapsto (x_0^{12},x_1^{12},x_2^3,x_3^2,x_0^6x_1^6,x_0^4x_1^4x_2,x_0^2x_1^2x_2^2,x_0x_1x_2x_3,x_0^3x_1^3x_3)
\end{align*}
and the image of this map is isomorphic to $\bW\bP(1,1,4,6)^{\circ}$. The anticanonical hypersurface defined by the equation above becomes the intersection of this image with the hypersurface
\begin{equation}\label{eq:torichypersurface} y_0 + b_0y_1+y_2+y_3+b_1y_4 + b_2y_5+b_3y_6+y_7+b_4y_8 = 0,\end{equation}
where $(y_0,\ldots,y_8)$ are homogeneous coordinates on $\bP^8$. Working on the maximal torus given by setting $y_8 = 1$ and taking all other $y_i \in \bC^*$, we can use relations on the image of $\psi$ to write
\[y_0 = \frac{1}{y_1y_3^2},\quad y_2 = \frac{y_7^3}{y_3},\quad y_4=\frac{1}{y_3},\quad y_5 = \frac{y_7}{y_3},  \quad y_6 = \frac{y_7^2}{y_3}.\]
Substituting in and clearing denominators we obtain
\[1 + b_0y_1^2y_3^2 + y_1y_3y_7^3 + y_1y_3^3 + b_1y_1y_3 + b_2y_1y_3y_7+b_3y_1y_3y_7^2 + y_1y_3^2y_7 + b_4y_1y_3^2 = 0.\]
Now set $x = y_7$, $y = y_3$, and $z = y_1y_3$ to get
\[1 + b_0z^2 + x^3z + y^2z + b_1z + b_2xz+b_3x^2z + xyz + b_4yz = 0.\]
Completing the square in $y$ allows us to eliminate the $xyz$ and $yz$ terms, then completing the cube in $x$ allows us to eliminate the $x^2z$ term, giving
\begin{equation} \label{eq:toricnormalform} 1 + b_0z^2 + x^3z + y^2z +c_1z + c_2xz = 0\end{equation}
for some constants $c_1, c_2 \in \bC$ defined as polynomial functions of the $b_i$. Adding a variable $w$ to homogenise to $\bP^3[x,y,z,w]$ and rescaling variables, we again obtain the projective normal form \eqref{eq:CDform}.

\subsection{Constructing Tyurin degenerations}\label{sec:E6E12}

In the previous sections we have given two constructions of the projective model \eqref{eq:CDform}. We now show how to construct an $M$-polarised Tyurin degeneration from this projective model and derive its properties.

The projective model \eqref{eq:CDform} has two singularities: a singularity of type $A_{11}$ at the point $(w,x,y,z) = (0,0,1,0)$ and a singularity of type $E_6$ at the point $(w,x,y,z) = (0,0,0,1)$. It also contains the two lines $\{w=x=0\}$, which contains both singularities, and $\{w=z=0\}$, which passes through the $A_{11}$ but not the $E_6$. These singularities may be crepantly resolved to give an $M$-polarised K3 surface and the $19$ curves from Figure \ref{fig:dualgraph} arise as follows: $E_1,\ldots,E_{11}$ and $E_{12},\ldots,E_{18}$ are the exceptional curves from the $A_{11}$ and $E_6$ singularities, respectively, $E_0$ is the strict transform of the line $\{w=z=0\}$, and $E_{12}$ is the strict transform of the line $\{w=x=0\}$.

By the discussion in Subsection \ref{sec:moduli}, Type II degenerations of \eqref{eq:CDform}  occur along the locus $\{d =0, a^3 \neq b^2\}$. To construct a Tyurin degeneration, fix values of $(a,b)$ with $a^3 \neq b^2$ and consider the family
\[\calX := \{y^2zw - 4x^3z+3axzw^2+bzw^3 - \tfrac{1}{2}(tz^2w^2+w^4) = 0\} \subset \bP^3[w,x,y,z] \times \Delta,\]
where $t \in \Delta$ is a parameter on the complex disc.

This family is singular along the two curves. Computations in Singular \cite{singular} reveal that the curve $\{w=x=z=0\}$ is a curve of $A_{11}$ singularities, whilst the curve $\{w=x=y=0\}$ consists of $E_6$ singularities for $t \neq 0$ which worsen to a singularity of Milnor number $8$ in the fibre $\{t=0\}$.

These singularities can be crepantly resolved as follows. Begin by blowing up the point $(w,x,y,z,t) = (0,0,0,1,0)$ in the ambient space. The strict transform of the threefold $\calX$ has central fibre consisting of two components (the strict transform of the original fibre $\{t=0\}$ plus an exceptional component) glued along a smooth elliptic curve, and is singular only along the strict transforms of the curves $\{w=x=z=0\}$ and $\{w=x=y=0\}$, which are now curves of $A_{11}$ and $E_6$ singularities respectively. These two curves can then be resolved by blowing up in the usual way.

From the description above it is easy to see that the $19$ curves $E_i$ on a general fibre which generate the lattice $M$ are invariant under monodromy, so the resulting degeneration is a projective $M$-polarised Tyurin degeneration. The strict transform of the original fibre $\{t=0\}$ is isomorphic to $V_{12}$ and the exceptional component is isomorphic to $V_6$, so the central fibre of this degeneration is the stable surface $(V_6,V_{12})$.

\section{Constructions}\label{sec:constructions}

Theorem \ref{thm:analyticclassification} shows that, up to labelling, there are $11$ possibilities for the central fibre of an $M$-polarised Tyurin degeneration. The following is the main result of this paper.

\begin{theorem} \label{thm:main} All of the $11$ possible central fibres from Theorem \ref{thm:analyticclassification} can be realised in projective $M$-polarised Tyurin degenerations. 
\end{theorem}

We prove this result by explicitly constructing projective models for each of the $11$ possible degenerations. The result is summarised in Table \ref{tab:gammadegenerations}, which lists the possible central fibres $X_0$, along with the methods used to construct the degenerations (i.e. the linear system method from Section \ref{sec:linearsystem} or the toric method from Section \ref{sec:toricapproach}); the projective model used in the construction, expressed as a hypersurface or complete intersection in weighted projective space, or a blow up of one of these; the singularities present in a generic projective model; and a reference to the section where the construction is given. 

\begin{table}
\begin{tabular}{|c|c|c|c|c|}
\hline 
$X_0$ & Construction & Model & Singularities & Section \\
\hline
$(V_9,V_9)$ & Linear system & $X_{2,6} \subset \bW\bP(1,1,1,2,3)$ & $E_8E_8A_1$ & \ref{sec:stdfib} \\
$(V_8,V_{10})$ & Linear system & $X_{12} \subset \bW\bP(1,1,4,6)$ & $E_8E_8A_1$ & \ref{sec:stdfib} \\
$(V_7,V_{11})$ & Toric & $X_{6} \subset \bW\bP(1,1,1,3)$ & $D_{10}E_7$& \ref{sec:toricchain}\\
$(V_6,V_{12})$& Both & $X_{4} \subset \bP^3$ &$A_{11}E_6$& \ref{sec:E6E12}\\
$(V_5,V_{13})$ & Toric & $X_{8} \subset \bW\bP(1,1,2,4)$ & $D_{12}D_5$ & \ref{sec:toricchain}\\
$(V_4,V_{14})$ & Toric & $X_{16} \subset \bW\bP(1,2,5,8)$ & $D_{13}A_4$ & \ref{sec:A4E14}  \\
$(V_3,V_{15})$ & Toric &$X_{10} \subset \bW\bP(1,1,3,5)$& $D_{14}A_2A_1$& \ref{sec:toricchain}\\
$(V_2,V_{16})$ & Linear system & $\mathrm{Bl}(X_{12} \subset \bW\bP(1,1,4,6))$ & $D_{15}A_1$ & \ref{sec:A1E16} \\
$(V_1,V_{17})$ & Linear system & $\mathrm{Bl}(X_{8} \subset \bW\bP(1,1,2,4))$  & $D_{16}$ & \ref{sec:E17} \\
$(V_0,V_{18})$ & Toric & $X_{6} \subset \bW\bP(1,1,1,3)$ & $A_{17}$& \ref{sec:E18} \\
$(\overline{V}_1,V_{17})$ & Linear system & $X_{12} \subset \bW\bP(1,1,4,6)$ & $D_{16}A_1$ & \ref{sec:altfib} \\
\hline

\hline
\end{tabular}
\caption{The $\Gamma$-degenerations, along with the methods used to construct them, the projective models used and their singularities, and section references for the constructions.}
\label{tab:gammadegenerations}
\end{table}

\subsection{Central fibres $(V_8,V_{10})$ and $(V_9,V_9)$}\label{sec:stdfib}

We construct these degenerations using the linear system approach of Section \ref{sec:linearsystem}. We first note that, given an elliptically fibred K3 surface $X$ with fibre $F$ and $(-2)$-section $S$, the $\bQ$-divisor $D = \frac{1}{2}S + F$ is nef and big and the canonical ring $R(X,D)$ is generated in degrees $\leq 6$. A straightforward Riemann-Roch calculation then shows that $\Proj R(X,D)$ can be realized as a hypersurface of degree $12$ in the weighted projective space $\bW\bP(1,1,4,6)$. The birational morphism $\varphi\colon X \to \Proj R(X,D)$ contracts $S$ along with those components of fibres that do not meet $S$; the image of this morphism is therefore the surface obtained by contracting the section in the Weierstrass model of $X$. By standard results in the theory of elliptic surfaces (see, for example, \cite{btes}), we obtain that the image $\varphi(X)$ is given by an equation of the form
\[w^2 = z^3 + a_8(x,y)z + b_{12}(x,y),\]
where $(x,y,z,w)$ are variables of weights $(1,1,4,6)$, respectively, and $a_8$ and $b_{12}$ are homogeneous of degrees $8$ and $12$.

\begin{remark} As elliptic fibrations with section on a K3 surface $X$ give rise to  primitive embeddings of $H$ into the Picard lattice of $X$, with $H$ generated by the classes of a section and fibre, this is effectively the statement that an $H$-polarised K3 can be expressed as the minimal resolution of a hypersurface of degree $12$ in $\bW\bP(1,1,4,6)$, which was already used in Section \ref{sec:toricapproach}.
\end{remark}

To construct a projective model for a $M$-polarised K3 surfaces we apply this approach to the standard fibration. Specifically, we take
\[D=(3;2,4,6,5,4,3,2,1,\tfrac{1}{2},0,0,0,0,0,0,0,0;0\mid 0,0).\]
Then we have the following sections, which generate $R(X,D)$:
\begin{align*}
x & = (3;2,4,6,5,4,3,2,1,0,0,0,0,0,0,0,0,0;0\mid 0,0) \in H^0(\calO_X(\lfloor D \rfloor)),\\
y&= (0;0,0,0,0,0,0,0,0,0,1,2,3,4,5,6,4,2;3\mid 0,0) \in H^0(\calO_X(\lfloor D \rfloor)),\\
z &= (5;4,7,10,8,6,4,2,0,0,0,2,4,6,8,10,7,4;5\mid 1,0) \in H^0(\calO_X(4D)),\\
w &= (8;5,10,15,12,9,6,3,0,0,0,3,6,9,12,15,10,5;8\mid 0,1) \in H^0(\calO_X(6D)).
\end{align*}
From the intersection properties of $D$ one can see that $\varphi$ contracts $E_0,\ldots,E_7$ and $E_{11},\ldots,E_{18}$ to a pair of $E_8$ singularities, which occur in the fibres $\{x=0\}$ and $\{y=0\}$, and contracts the section $E_9$ to an $A_1$ singularity. The curves $E_8$ and $E_{10}$ are taken to $\{x=0\}$ and $\{y = 0\}$ respectively. Standard results on the classification of fibre types in Weierstrass models show that, after a coordinate change, a generic surface satisfying these conditions is given by
\begin{equation}\label{eq:stdweierstrass} \{w^2 = z^3 + a_1x^4y^4z + a_2x^5y^7 + a_3x^6y^6 + x^7y^5\} \subset \bW\bP(1,1,4,6),\end{equation}
for constants $a_i \in \bC$. It is easy to show that a minimal resolution for such a surface is a K3 containing the $19$ curves $E_i$ which generate $M$, so this is a projective model for an $M$-polarised K3 surface.

A Tyurin degeneration is obtained by fixing values of $a_1,a_3 \in \bC$ and letting $a_2 = t$ be a parameter on the complex disc $\Delta$. The resulting degeneration has three curves of singularities: $\{x = y = w^2-z^3 = 0\}$, which is a curve of $A_1$'s, $\{w=y=z=0\}$, which is a curve of $E_8$'s, and $\{w=x=z=0\}$ which consists of $E_8$'s for $t \neq 0$ and worsens to an elliptic singularity of type $\widetilde{E}_8$ in the fibre $\{t = 0\}$.

These singularities can be crepantly resolved as follows. Blow up the point $(w,x,z,t) = (0,0,0,0)$ with weights $(3,1,2,1)$ respectively. The strict transform then has central fibre consisting of two components glued along a smooth elliptic curve, and is singular only along curves of $E_8$, $E_8$, and $A_1$ singularities, which can be resolved in the usual way. The result is an $M$-polarised Tyurin degeneration whose central fibre is a stable surface $(V_8,V_{10})$.

A second Tyurin degeneration, which differs from the $E_8E_{10}$-degeneration above by a single flop, can be constructed from the projective model \eqref{eq:stdweierstrass} as follows. Perform a Veronese embedding of degree $2$ to obtain the complete intersection
\[\{w^2 - z^3 - a_1u_2^4z - a_2u_2^5u_3 - a_3u_2^6 - u_1u_2^5 = u_1u_3-u_2^2 = 0\} \subset \bW\bP(1,1,1,2,3),\]
where $\bW\bP(1,1,1,2,3)$ has coordinates $u_1 = x^2$, $u_2 = xy$, $u_3 = y^2$, $z$ and $w$. Now consider the degeneration given by
\[\{w^2 - z^3 - a_1u_2^4z -a_2u_2^5u_3 - a_3u_2^6 - u_1u_2^5 = u_1u_3-tu_2^2 = 0\} \subset \bW\bP(1,1,1,2,3) \times \Delta.\]
The central fibre of this degeneration consists of two components, one with $u_1 = 0$ and one with $u_3 = 0$, and the threefold total space is singular along three curves of singularities: two of type $E_8$ and one of type $A_1$. After crepantly resolving each of these curves we obtain an $M$-polarised Tyurin degeneration whose central fibre is a stable surface $(V_9,V_{9})$.

\subsection{Central fibre $(\overline{V}_1,V_{17})$}\label{sec:altfib}

We can obtain another projective model by applying the approach of Section \ref{sec:stdfib} to the alternate fibration, with $E_1$ as the distinguished $(-2)$-section. We take
\[D=(1;\tfrac{1}{2},1,2,2,2,2,2,2,2,2,2,2,2,2,2,1,0;1\mid 0,0).\]
Then we have the following sections, which generate $R(X,D)$:
\begin{align*}
x & = (1;0,1,2,2,2,2,2,2,2,2,2,2,2,2,2,1,0;1\mid 0,0) \in H^0(\calO_X(\lfloor D \rfloor)),\\
y&= (0;0,0,0,0,0,0,0,0,0,0,0,0,0,0,0,0,0;0\mid 1,0) \in H^0(\calO_X(\lfloor D \rfloor)),\\
z &= (1;0,0,2,3,4,5,6,7,8,9,10,11,12,13,14,8,2;7\mid 0,0) \in H^0(\calO_X(4D)),\\
w &= (2;0,0,3,4,5,6,7,8,9,10,11,12,13,14,15,8,1;8\mid 0,1) \in H^0(\calO_X(6D)).
\end{align*}
The morphism $\varphi$ contracts $E_0,E_3,\ldots,E_{16},E_{18}$ to a $D_{16}$ singularity, which occurs at the point $(x,y,z,w) = (0,1,0,0)$, and contracts the section $E_1$ to an $A_1$ singularity, which occurs at $(x,y,z,w) = (0,0,1,1)$. The curves $E_2$ and $E_{17}$ are taken to $\{x=w^2-z^3=0\}$ and $\{z = w= 0\}$ respectively. Standard results on the classification of fibre types in Weierstrass models show that, after a coordinate change, a generic surface satisfying these conditions is given by
\begin{equation}\label{eq:altweierstrass} \{w^2 = z^3 + (a_1x^4+a_2x^3y+xy^3)z^2 + a_3x^8z\} \subset \bW\bP(1,1,4,6),\end{equation}
for constants $a_i \in \bC$. It is easy to show that a minimal resolution for such a surface is a K3 containing the $19$ curves $E_i$ which generate $M$, so this is a projective model for an $M$-polarised K3 surface.

A Tyurin degeneration is obtained by fixing values of $a_1,a_2 \in \bC$ and letting $a_3 = t$ be a parameter on the complex disc $\Delta$. This family contains a curve of generically $A_1$ singularities given by $\{t = z = w = 0\}$. Blowing up this curve once, one obtains a degeneration whose central fibre contains two components: the strict transform of the original central fibre is a double cover of $\bW\bP(1,1,4)$ ramified along the smooth quartic curve $\{z+a_1x^4+a_2x^3y+xy^3 = 0\}$, which contains a singularity of type $A_1$ over the $\frac{1}{4}(1,1)$ orbifold point, and the exceptional component is a double cover of the Hirzebruch surface $\bF_4$ ramified in two curves in the classes $(s + 4f)$, and $(s+8f)$ (where $s$ and $f$ denote the classes of the $(-4)$-section and fibre on $\bF_4$ respectively), meeting at a singularity of type $D_{16}$. These components meet along a smooth elliptic curve given as a double cover of the line $\{z=0\} \subset \bP(1,1,4)$ and the $(-4)$-section in $\bF_4$.

The resulting degeneration is singular only in the curves of $D_{16}$ and $A_1$ singularities, which can be crepantly resolved in the usual way. The result is an $M$-polarised Tyurin degeneration whose central fibre is a stable surface $(\overline{V}_1,V_{17})$.

\subsection{Central fibres  $(V_7,V_{11})$, $(V_5,V_{13})$, and $(V_3,V_{15})$} \label{sec:toricchain}

These degenerations are all closely related and it makes sense to construct them together. We apply the toric approach of Section \ref{sec:toricapproach}.

With notation as in Section \ref{sec:toricapproach}, we work on the maximal torus $y_4 = 1$ and take all other $y_i \in \bC^*$, then use relations on the image of $\psi$ to write
\[y_1 = \frac{1}{y_0},\quad y_2=y_5^3,\quad y_3=y_8^2,\quad y_6 = y_5^2,\quad y_7 = y_5y_8.\]
Substituting into Equation \eqref{eq:torichypersurface} and multiplying by $y_0^2$ we obtain
\[y_0^3 + b_0y_0 + y_0^2y_5^3 + y_0^2y_8^2 + b_1y_0^2 + b_2 y_0^2y_5 + b_3y_0^2y_5^2 + y_0^2y_5y_8 + b_4y_0^2y_8 = 0.\]
Now set $y = y_5$, $z = y_0$, and $w = y_0y_8$ to get
\[z^3 + b_0z + y^3z^2 + w^2 + b_1z^2 + b_2 yz^2 + b_3 y^2z^2 + wyz + b_4wz = 0.\]
Completing the square in $w$ allows us to eliminate the $wyz$ and $wz$ terms, the completing the cube in $y$ allows us to eliminate the $y^2z^2$ term, giving
\begin{equation}z^3 + a_1z + y^3z^2 + w^2 + a_2z^2 + a_3yz^2 = 0,\label{eq:toricchain}\end{equation}
for some constants $a_i \in \bC$.

Now we make an important observation: by adjusting the weights, there are three different ways to homogenize this equation to define a hypersurface in weighted projective space. Specifically, if we take $(y,z,w)$ to have weights $(1,n,n+2)$, respectively, for $n \in \{1,2,3\}$, then we can homogenize by adding a variable $x$ of weight $1$ to obtain a hypersurface
\[\{w^2+xz(a_1x^{n+3}+x^{3-n}z^2+a_2x^3z+a_3x^2yz+y^3z) = 0\} \subset \bW\bP(1,1,n,n+2).\]
This hypersurface contains a $D_{8+2n}$ singularity at $(w,x,y,z) = (0,1,0,0)$ and an $E_{9-2n}$ singularity at $(w,x,y,z) = (0,0,1,0)$ (where we use the convention that $E_5 = D_5$ and $E_3 = A_2 + A_1$). The resolution is an $M$-polarised K3 surface, where $E_0,E_2,\ldots,E_{8+2n}$ are the exceptional curves coming from the $D_{8+2n}$ singularity, $E_{10+2n},\ldots,E_{18}$ are the exceptional curves coming from the $E_{9-2n}$ singularity, and $E_1$ and $E_{9+2n}$ are the strict transforms of the lines $\{z=w=0\}$ and $\{x = w=0\}$, respectively. 

\begin{remark}
If $\varphi_n\colon X \to \bW\bP(1,1,n,n+2)$ is the projective model for an $M$-polarised K3 surface $X$ described above (for $n \in \{1,2,3\}$), we may describe the pull-back $\varphi_n^*\calO(1)$ as a $\bQ$-divisor $D_n$ on $X$ and give generating sections of the canonical ring $R(X,D_n)$ corresponding to the coordinates $(x,y,z,w)$. Using this, one may check the description above using the linear system method from Section \ref{sec:linearsystem}.  Explicitly, one obtains
\begin{align*}D_1&= (1;0,1,2,2,2,2,2,2,2,2,2,3,4,5,6,4,2;3\mid 0,0), \\
D_2 &= (1;0,1,2,2,2,2,2,2,2,2,2,2,2,\tfrac{5}{2},3,2,1;\tfrac{3}{2}\mid 0,0),\\
D_3 &= (1;0,1,2,2,2,2,2,2,2,2,2,2,2,2,2,\tfrac{4}{3},\tfrac{1}{3};1\mid 0,0).
\end{align*}
Then $R(X,D_1)$ is generated by
\begin{align*}
x&= (1;0,1,2,2,2,2,2,2,2,2,2,3,4,5,6,4,2;3\mid 0,0) \in H^0(\calO_X(D_1)),\\
y &= (0;0,0,0,0,0,0,0,0,0,0,0,1,2,3,4,3,2;2\mid 1,0) \in H^0(\calO_X(D_1)),\\
z&= (4;2,5,8,7,6,5,4,3,2,1,0,0,0,0,0,0,0;0\mid 0,0) \in H^0(\calO_X(D_1)),\\
w &= (5;1,5,9,8,7,6,5,4,3,2,1,3,5,7,9,6,3;5\mid 0,1) \in H^0(\calO_X(3D_1));
\end{align*}
$R(X,D_2)$ is generated by
\begin{align*}
x&= (1;0,1,2,2,2,2,2,2,2,2,2,2,2,2,3,2,1;1\mid 0,0) \in H^0(\calO_X(\lfloor D_2 \rfloor)),\\
y & = (0;0,0,0,0,0,0,0,0,0,0,0,0,0,0,1,1,1;0\mid 1,0) \in H^0(\calO_X(\lfloor D_2 \rfloor)),\\
z &= (5;2,6,10,9,8,7,6,5,4,3,2,1,0,0,0,0,0;0\mid 0,0) \in H^0(\calO_X(2D_2)),\\
w &= (6;1,6,11,10,9,8,7,6,5,4,3,2,1,2,3,2,1;2\mid 0,1) \in H^0(\calO_X(4D_2));
\end{align*}
and $R(X,D_3)$ is generated by
\begin{align*}
x&= (1;0,1,2,2,2,2,2,2,2,2,2,2,2,2,2,1,0;1\mid 0,0) \in H^0(\calO_X(\lfloor D_3 \rfloor)),\\
y & = (0;0,0,0,0,0,0,0,0,0,0,0,0,0,0,0,0,0;0\mid 1,0) \in H^0(\calO_X(\lfloor D_3 \rfloor)),\\
z &= (6;2,7,12,11,10,9,8,7,6,5,4,3,2,1,0,0,0;0\mid 0,0) \in H^0(\calO_X(3D_3)),\\
w &= (7;1,7,13,12,11,10,9,8,7,6,5,4,3,2,1,0,0;1\mid 0,1) \in H^0(\calO_X(\lfloor 5D_3\rfloor)).
\end{align*}
\end{remark}

A Tyurin degeneration is obtained by fixing values of $a_2,a_3 \in \bC$ and letting $a_1 = t$ be a parameter on the complex disc $\Delta$. This family contains a curve of generically $A_1$ singularities given by $\{t = z = w = 0\}$. Blowing up this curve once, one obtains a degeneration whose central fibre contains two components: the strict transform of the original central fibre is a double cover of $\bW\bP(1,1,n)$ ramified along the quartic curve $\{x(x^{2-n}z+a_2x^3+a_3x^2y+y^3) = 0\}$, which contains a singularity of type $E_{9-2n}$, and the exceptional component is a double cover of the Hirzebruch surface $\bF_n$ ramified in three curves in the classes $f$, $(s + nf)$, and $(s+(n+3)f)$ (where $s$ and $f$ denote the classes of the $(-n)$-section and fibre on $\bF_n$ respectively), meeting at a singularity of type $D_{8+2n}$. These components meet along a smooth elliptic curve given as a double cover of the line $\{z=0\} \subset \bP(1,1,n)$ and the $(-n)$-section in $\bF_n$.

The resulting degeneration is singular only in the curves of $D_{8+2n}$ and $E_{9-2n}$ singularities, which can be resolved in the usual way.  The result is an $M$-polarised Tyurin degeneration whose central fibre is a stable surface $(V_7,V_{11})$ for $n=1$, a stable surface $(V_5,V_{13})$ for $n=2$, and a stable surface $(V_3,V_{15})$ for $n=3$.

\begin{remark} The approach from this section can also be applied with $n=4$ to yield the degeneration with central fibre $(\overline{V}_1,V_{17})$ from Section \ref{sec:altfib}. However, given the relationship between this degeneration and the alternate fibration, and because it is geometrically of a slightly different character to the degenerations in this section (owing to the fact that the branch divisor is not divisible by $x$), we chose to keep it in its own section.
\end{remark}

\subsection{Central fibre $(V_4,V_{14})$}\label{sec:A4E14}

The approach from Section \ref{sec:toricchain} can be used to construct one further case. Starting from the affine equation \eqref{eq:toricchain},
\[z^3 + a_1z + y^3z^2 + w^2 + a_2z^2 + a_3yz^2 = 0,\]
we can take $(y,z,w)$ to have weights $(2,5,8)$, respectively, then homogenize by adding a variable $x$ of weight $1$ to obtain a hypersurface
\[\{w^2+z(a_1x^{11}+a_2x^6z+a_3x^4yz+xz^2+y^3z) = 0\} \subset \bW\bP(1,2,5,8).\]
This hypersurface contains a $D_{13}$ singularity at $(x,y,z,w) = (0,1,0,0)$ and an $A_4$ singularity at $(x,y,z,w) = (0,0,1,0)$. The resolution is an $M$-polarised K3 surface, where $E_0,E_2,\ldots,E_{13}$ are the exceptional curves coming from the $D_{13}$ singularity, $E_{15},\ldots,E_{18}$ are the exceptional curves from the $A_4$, and $E_1$ and $E_{14}$ are the strict transforms of the curves $\{z=w=0\}$ and $\{x = w^2+y^3z^2 = 0\}$ respectively.

\begin{remark}
If $\varphi\colon X \to \bW\bP(1,2,5,8)$ is the projective model for an $M$-polarised K3 surface $X$ described above, we may describe the pull-back $\varphi^*\calO(1)$ as a $\bQ$-divisor $D$ on $X$ and give generating sections of the canonical ring $R(X,D)$ corresponding to the coordinates $(x,y,z,w)$. Using this, one may check the description above using the linear system method from Section \ref{sec:linearsystem}.  Explicitly, one obtains
\[D=(\tfrac{1}{2};0,\tfrac{1}{2},1,1,1,1,1,1,1,1,1,1,1,1,\tfrac{6}{5},\tfrac{4}{5}, \tfrac{2}{5};\tfrac{3}{5}\mid 0,0)\]
and we have the following sections, which generate $R(X,D)$:
\begin{align*}
x & = (0;0,0,1,1,1,1,1,1,1,1,1,1,1,1,1,0, 0;0\mid 0,0) \in H^0(\calO_X(\lfloor D \rfloor)),\\
y&= (0;0,0,0,0,0,0,0,0,0,0,0,0,0,0,0,0, 0;0\mid 1,0) \in H^0(\calO_X(\lfloor 2D \rfloor)),\\
z &= (5;2,5,11,10,9,8,7,6,5,4,3,2,1,0,0,0,0;0\mid 0,0) \in H^0(\calO_X(\lfloor 5D \rfloor)),\\
w &= (6;1,6,11,10,9,8,7,6,5,4,3,2,1,0,0,0,0;0\mid 0,1) \in H^0(\calO_X(\lfloor 8D \rfloor)).
\end{align*}
\end{remark}

A Tyurin degeneration is obtained by fixing values of $a_2,a_3 \in \bC$ and letting $a_1 = t$ be a parameter on the complex disc $\Delta$. This family contains a curve of generically $A_1$ singularities given by $\{t = z = w = 0\}$. Blowing up this curve once, one obtains a degeneration whose central fibre contains two components: the strict transform of the original central fibre is a double cover of $\bW\bP(1,2,5)$ ramified along the degree $6$ curve $\{a_2x^6 + a_3x^4y+ xz + y^3  = 0\}$, which contains a singularity of type $A_4$ over the point $(0,0,1)$. 

The exceptional component is most simply described as a double cover of a \emph{half Hirzebruch surface}. The following definition is fairly well-known.

\begin{definition} The half Hirzebruch surface $\bF_{n+\frac{1}{2}}$ is defined as follows. First blow up the Hirzebruch surface $\bF_n$ at a point on the $(-n)$-section, then blow up again at the intersection between the two $(-1)$-curves in the reducible fibre that results; the resulting surface has a $\bP^1$-fibration with a single reducible fibre consisting of three curves with self-intersections $(-2,-1,-2)$ and multiplicities $(1,2,1)$. Finally, contract the two $(-2)$-curves to a pair of $A_1$ singularities. The resulting singular surface is $\bF_{n+\frac{1}{2}}$.  

If $s$ and $f$ are the strict transforms of the $(-n)$-section and a generic fibre from $\bF_n$, then the Weil divisor class group of $\bF_{n+\frac{1}{2}}$ is generated by $s$ and $\frac{1}{2}f$, with intersection numbers $s^2 = -n-\frac{1}{2}$, $s.f = 1$, and $f^2 = 0$.\end{definition}

With this definition, the exceptional component is a double cover of the half Hirzebruch surface $\bF_{\frac{5}{2}}$ ramified over two curves in the classes $(s + \frac{5}{2}f)$ and $(s+\frac{11}{2}f)$, chosen so that the double cover has a singularity of type $D_{13}$. This component meets the strict transform of the original central fibre along a smooth elliptic curve, given as a double cover of the line $\{z=0\} \subset \bP(1,2,5)$ and the $(-\frac{5}{2})$-section in $\bF_{\frac{5}{2}}$.

The resulting degeneration is singular only in the curves of $D_{13}$ and $A_4$ singularities, which can be resolved in the usual way. The result is an $M$-polarised Tyurin degeneration whose central fibre is a stable surface $(V_4,V_{14})$.

\begin{remark} We note that this method can also be used to construct an alternative model for the degeneration with central fibre $(V_6,V_{12})$ from Sections \ref{sec:linearsystem}, \ref{sec:toricapproach}, and \ref{sec:E6E12}, by assigning weights $(2,3,6)$ to $(y,z,w)$ and compactifying to $\bW\bP(1,2,3,6)$. We chose to retain the construction given in those sections for compatibility with the existing literature on $M$-polarised K3 surfaces, which makes extensive use of the normal form \eqref{eq:CDform}.
\end{remark}

\subsection{Central fibre $(V_0,V_{18})$}\label{sec:E18}

To construct this degeneration we return to the affine equation \eqref{eq:toricnormalform}:
\[ 1 + b_0z^2 + x^3z + y^2z +c_1z + c_2xz = 0.\]
In Section \ref{sec:toricapproach} we obtained a projective model by compactifying this to a hypersurface in $\bP^4$. However, as we saw in Section \ref{sec:toricchain}, we can obtain different projective models by compactifying to different weighted projective spaces.

We can compactify the equation above to a sextic hypersurface in $\bW\bP(1,1,1,3)$ by assigning weights $(1,1,3)$ to $(x,y,z)$ and introducing a new weight $1$ variable $w$:
\[w^6 + b_0z^2 + x^3z + wy^2z + c_1w^3z+ c_2w^2xz = 0.\]
By the results of Sections \ref{sec:moduli} and \ref{sec:toricapproach}, it is easy to see that the parameters $(b_0,c_1,c_2)$ in the equation above are related to the modular parameters $(a,b,d)$ describing an $M$-polarised K3 surface by
\[ b_0 = d, \qquad c_1 = 2b, \qquad c_2 = -3a,\]
from which it follows that $b_0 \neq 0$ for any $M$-polarised K3 surface. We can therefore make a coordinate change $z \mapsto \frac{z}{b_0}$ and clear denominators to obtain
\[b_0w^6 + z^2 + x^3z + wy^2z + c_1w^3z + c_2w^2xz = 0.\]
Completing the square in $z$ and rearranging gives the projective model
\[\{z^2 + ((b_1+c_1)w^3 + x^3 + wy^2 + c_2w^2x)((b_1-c_1)w^3-x^3-wy^2-c_2w^2x) = 0\}\]
in $\bW\bP(1,1,1,3)$, where $b_1 \in \bC$ is a square root of $4b_0$.

This model is a double cover of $\bP^2$ ramified over the two cubic curves:
\begin{align*}
\{w^2y + x^3 + c_2w^2x + (c_1+b_1)w^3 &= 0\},\\
\{w^2y + x^3 + c_2w^2x + (c_1-b_1)w^3 &= 0\}.
\end{align*}
These two curves meet in the single point $(w,x,y) = (0,0,1)$, over which the double cover has a singularity of type $A_{17}$. The resolution is an $M$-polarised K3 surface, where $E_1,\ldots,E_{17}$ are the exceptional curves coming from the resolution of the $A_{17}$ singularity and $E_0$, $E_{18}$ are the curves $\{w= z+x^3 = 0\}$ and $\{w = z-x^3 = 0\}$, respectively, so this is a projective model.

A Tyurin degeneration is obtained by fixing values of $c_1,c_2 \in \bC$ and letting $b_1 = t$ be a parameter on the complex disc $\Delta$. The central fibre of this family is a union of the two surfaces
\begin{align*}
\overline{Y}_1 := \{z + c_1w^3 + x^3 + wy^2 + c_2w^2x &= 0\} \subset \bW\bP(1,1,1,3),\\
\overline{Y}_2 := \{z - c_1w^3 - x^3 - wy^2 - c_2w^2x &= 0\} \subset \bW\bP(1,1,1,3),
\end{align*}
each of which is isomorphic to $\bP^2$ and is a non-$\bQ$-Cartier divisor in the total space of the degeneration. The curves $E_0$ and $E_{18}$ degenerate to $\{w = z+x^3 = 0\} \subset \overline{Y}_1$ and $\{w = z-x^3 = 0\}\subset \overline{Y}_2$, respectively. 

To resolve the singularities of this degeneration, first blow-up the locus $\{t = z + c_1w^3 + x^3 + wy^2 + c_2w^2x = 0\}$ in the total space of the family; this corresponds to blowing along a non-$\bQ$-Cartier divisor. The strict transform $Y_1$ of $\overline{Y}_1$ is the blow-up of $\overline{Y}_1$ at the single point $\{w=x=z=0\}$, and the strict transform $Y_2$ of $\overline{Y}_2$ is isomorphic to $\overline{Y}_2$.

After performing this blow-up, the family is singular only along a smooth curve of $A_{17}$ singularities, which intersects the central fibre in a point of $Y_1$ away from the double curve. Resolving this curve in the usual way, we obtain an $M$-polarised Tyurin degeneration whose central fibre is a stable surface $(V_0,V_{18})$. The inflectional tangent line $F_0 \subset V_0 \cong Y_2$ is the curve $\{w = z-x^3 = 0\}$.

\subsection{Central fibre $(V_1,V_{17})$}\label{sec:E17}

The final two cases are somewhat more technical, as they are not constructed directly as hypersurfaces in weighted projective space; instead, we start with a hypersurface in weighted projective space and perform a single blow-up. 

To construct a degeneration with central fibre $(V_1,V_{17})$ we use the linear system approach from Section \ref{sec:linearsystem}. We begin with the nef and big $\bQ$-divisor
\[D = (1;1,2,3,3,3,3,3,3,3,3,3,3,3,3,3,\tfrac{3}{2}, 0;\tfrac{3}{2}\mid 0,0).\]
By Riemann-Roch, $H^0(\calO_X(\lfloor D \rfloor))$ is generated by the two sections
\begin{align*}
x & = (1;1,2,3,3,3,3,3,3,3,3,3,3,3,3,3,1, 0;1\mid 0,0) \in H^0(\calO_X(\lfloor D \rfloor)),\\
y&= (1;1,1,1,1,1,1,1,1,1,1,1,1,1,1,1,0, 0;0\mid 1,0) \in H^0(\calO_X(\lfloor D \rfloor)).
\end{align*}
Now, $2D$ is a Cartier divisor and $h^0(\calO_X(2D)) = 4$, but we only obtain three generators $\{x^2,xy,y^2\}$ from above. Moreover, we cannot obtain any extra generators of $H^0(\calO_X(2D))$ using the relations \eqref{eq:Srelation} and \eqref{eq:Trelation}. However, as the linear system $|2D|$ is not an elliptic pencil and $D$ is nef and big, by \cite[Propositions 1 and 8]{fk3s} we have that the generic member of $|2D|$ is a smooth, irreducible curve. Let $Z$ be any such curve and let $z \in H^0(\calO_X(2D))$ be its defining section. The curve $Z$ intersects our divisors $E_i$, $S$ and $T$ as follows:
\[Z.E_0 = 2, \quad Z.E_{17} = 3, \quad Z.S = 2, \quad Z.T=5, \quad Z.E_i = 0 \text{ for } i \neq 0,17.\]
Then $\{x^2,xy,y^2,z\}$ generate $H^0(\calO_X(2D))$, and one can show that the canonical ring is generated by $x,y,z$ and
\[w = (0;1,2,3,4,5,6,7,8,9,10,11,12,13,14,15,8,1;8\mid 0,1) \in H^0(\calO_X(4D)).\]
It follows that $\Proj R(X,D)$ is isomorphic to $\bW\bP(1,1,2,4)$ and the image $\varphi(X)$ is defined by an equation of degree $8$.

\begin{remark}\label{rem:f3} It follows from this computation that the divisor
\[ (0;1,2,3,4,5,6,7,8,9,10,11,12,13,14,15,8,2;7\mid 0,0) \in H^0(\calO_X(\lfloor 3D \rfloor)),\]
obtained by applying the relation \eqref{eq:E8relation} to $x^3$, can be written as a linear combination of the generators $\{x^3,x^2y,xy^2,y^3,xz,xy\}$. We denote this linear combination by $f_3(x,y,z)$.
\end{remark}

From the intersection properties of $D$ one can see that $\varphi$ contracts the curves $E_1,\ldots,E_{16},E_{18}$ to a $D_{17}$ singularity, which occurs at the point $(x,y,z,w) = (0,0,1,0)$. The curve $E_{17}$ is taken to $\{w = f_3(x,y,z) = 0\}$, where $f_3$ is the same as in Remark \ref{rem:f3}, and $E_0$ is taken to the intersection of the hyperplane $\{x=0\}$ with our projective model.

After completing the square in $w$, we may assume that our projective model is a double cover of $\bW\bP(1,1,2)$ given by an equation of the form
\[w^2 = f_{3}(x,y,z)g_5(x,y,z),\]
where $g_5$ is an equation of degree $5$ defining the image of the divisor $T$. Moreover, the intersection properties of $D$ show that the two curves $\{f_3(x,y,z) = 0\}$ and $\{g_5(x,y,z) = 0\}$ are both tangent to $\{x=0\}$ at $(x,y,z) = (0,0,1)$, and they do not intersect away from this point.

After a coordinate change in $z$, we may assume that $f_3(x,y,z) = xz + cy^3$, for some nonzero $c \in \bC$. From the intersection condition between $f_3$ and $g_5$, it follows that $g_5$ must have the form
\[g_5(x,y,z) = (xz + cy^3)(a_0z + a_1y^2 + a_2xy + a_3x^2) + a_4x^5,\]
for $a_i \in \bC$ constants with $a_0,a_4 \neq 0$. After rescaling $x$, $y$, $z$ we may assume that $c = a_0 = 1$, so our projective model is given by
\[w^2 = (xz+y^3) \big((xz + cy^3)(a_0z + a_1y^2 + a_2xy + a_3x^2) + a_4x^5\big).\]
One may verify that the minimal resolution of such a surface is a K3 containing the $19$ curves $E_i$ which generate $M$.

To construct a degeneration we begin by partially resolving the $D_{17}$ singularity. To do this we blow-up the model above once at the point $(0,0,1,0)$, so that it becomes a double cover of the Hirzebruch surface $\bF_2$ instead of the cone $\bW\bP(1,1,2)$. The Hirzebruch surface $\bF_2$ can be realised explicitly as the rational scroll $\bF(2,0)$, which is constructed as the quotient of $\left(\bA^2 \setminus \{0\}\right)^2$ by the action of $(\bC^*)^2$ defined by
\begin{align*}
(\lambda,1) \colon (x,y;u,v) &\longmapsto (\lambda x, \lambda y;\lambda^{-2}u,v)\\
(1,\mu) \colon (x,y;u,v) &\longmapsto (x, y;\mu u, \mu v),
\end{align*}
where $(x,y;u,v)$ are coordinates on $\left(\bA^2 \setminus \{0\}\right)^2$ and $(\lambda,\mu) \in (\bC^*)^2$. For generalities on rational scrolls, we refer the reader to \cite[Chapter 2]{chapters}.

After this blow-up, our projective model becomes a double cover of $\bF_2$ ramified over the $(-2)$-section $s$ and two curves in the linear systems $|s+3f|$ and $|s+5f|$, where $f$ is the class of a fibre of the ruling on $\bF_2$. In the bihomogeneous coordinates $(x,y;u,v)$ introduced in the rational scroll description above, this branch curve is given explicitly by
\[u(vx+uy^3) \big((vx + uy^3)(v + a_1uy^2 + a_2uxy + a_3ux^2) + a_4u^2x^5\big).\]
Note that it has a singularity of type $D_{16}$ at the point $(x,y;u,v) = (0,1;0,1)$.

A Tyurin degeneration is obtained by fixing values of $a_1,a_2,a_3 \in \bC$ and letting $a_4 = t$ be a parameter on the complex disc $\Delta$. This family contains a curve of generically $A_1$ singularities given by $\{t = w = vx + uy^3 = 0\}$. Blowing up this curve once, one obtains a degeneration whose central fibre contains two components: the strict transform of the original central fibre is a double cover of $\bF_2$ ramified over the $(-2)$-section $\{u=0\}$ and the curve $\{v + a_1uy^2 + a_2uxy + a_3ux^2 = 0\} \in |s+2f|$, and the exceptional component is a double cover of the Hirzebruch surface $\bF_4$ ramified in three curves in the classes $f$, $(s + 4f)$, and $(s+7f)$ (where $s$ and $f$ denote the classes of the $(-4)$-section and fibre on $\bF_4$ respectively), meeting at a singularity of type $D_{16}$. These components meet along a smooth elliptic curve given as a double cover of the curve $\{vx + uy^3 = 0\} \subset \bF_2$ and the $(-4)$-section in $\bF_4$.

The resulting degeneration is singular only in the curve of $D_{16}$ singularities, which can be resolved in the usual way. The result is an $M$-polarised Tyurin degeneration whose central fibre is a stable surface $(V_1,V_{17})$.

\subsection{Central fibre $(V_2,V_{16})$}\label{sec:A1E16}

To construct the final projective Tyurin degeneration, we begin with the projective model from Section \ref{sec:altfib}, given by Equation \eqref{eq:altweierstrass}:
\[\{w^2 = z^3 + (a_1x^4+a_2x^3y+xy^3)z^2 + a_3x^8z\} \subset \bW\bP(1,1,4,6).\]
Recall that this model is a double cover of $\bW\bP(1,1,4)$ branched along the divisor
\[B= \{z^3 + (a_1x^4+a_2x^3y+xy^3)z^2 + a_3x^8z\} \subset \bW\bP(1,1,4),\]
which contains a singularity of type $D_{16}$ at the point $(x,y,z) = (0,1,0)$.

As in Section \ref{sec:E17}, to construct our degeneration we begin by partially resolving the $D_{16}$ singularity. Explicitly, this is given by performing a weighted blow-up of $\bW\bP(1,1,4)$ along the ideal $\langle xy^3 + z, z^2 \rangle$; i.e. $z$ has weight $1$ and $xy^3-z$ has weight $2$. Our projective model is a double cover of the resulting surface $S$, branched along the strict transform of the divisor $B$.

In the affine chart $y=1$ we can describe this blow-up explicitly as follows. Introduce a new variable $s := x+z$ and use it to eliminate $x$. The equation of $B$ becomes
\[\{z(a_1(s-z)^4+a_2(s-z)^3+s)z+a_3(s-z)^8)=0\} \subset \bA^2[s,z].\]
Now perform a weighted blow-up along the ideal $\langle s,z^2 \rangle$. The resulting surface $S$ is given by the equation
\[ \{sv^2 = z^2u\} \subset \bA^2[s,z] \times \bW\bP(1,2),\]
where $(u,v)$ are variables of weights $(2,1)$ on $\bW\bP(1,2)$, and the strict transform of $B$ is the intersection of this surface with
\[\big\{(s-z)(su-2zu+v^2)\big(a_1(s-z)v^2+a_2v^2 + a_3z(s-z)^3(su-2zu+v^2)\big) + uv^2  = 0\big\}.\]
The double cover branched along $B$ has a singularity of type $D_{15}$ over the point $s=z=v=0$.

A Tyurin degeneration is obtained by fixing values of $a_1,a_2 \in \bC$ and letting $a_3 = t$ be a parameter on the complex disc $\Delta$. This family contains a curve of generically $A_1$ singularities, given in the chart above by $\{t = z = v = 0\}$. Blowing up this curve once, one obtains a degeneration whose central fibre contains two components: the strict transform of the original central fibre is a double cover of $S$ ramified over a curve given in the chart above by $\{(s-z)(su-2zu+v^2)(a_1(s-z)+a_2) + u  = 0\}$, and the exceptional component is a double cover of the half Hirzebruch surface $\bF_{\frac{7}{2}}$ ramified in two curves in the classes $(s + \frac{7}{2}f)$ and $(s+\frac{13}{2}f)$ (where $s$ and $f$ denote the classes of the $(-\frac{7}{2})$-section and a generic fibre on $\bF_{\frac{7}{2}}$ respectively), chosen so that the double cover has a singularity of type $D_{15}$. These components meet along a smooth elliptic curve given as a double cover of the curve $\{z=v = 0\} \subset S$ and the $(-\frac{7}{2})$-section in $\bF_{\frac{7}{2}}$.

The resulting degeneration is singular only in the curve of $D_{15}$ singularities, which can be resolved in the usual way. The result is an $M$-polarised Tyurin degeneration whose central fibre is a stable surface $(V_2,V_{16})$.

\bibliography{publications,preprints}
\bibliographystyle{amsalpha}
\end{document}